%% file: Fano-HAL-final.tex
\documentclass[11pt,twoside]{article}

\usepackage{fancyhdr}
\usepackage{amsmath,amssymb,amsthm}
\usepackage{natbib}
\usepackage{cases}
\usepackage{tab}
\usepackage[paperwidth=210mm,
                paperheight=297mm,
                top=2.5cm,
                textheight=247mm,
                footnotesep=2\baselineskip,
                left=2cm,
                right=2cm,
                marginparwidth=1.7cm,
                headsep=20pt,
                dvips=false]{geometry}
\usepackage{dsfont}
\usepackage{bigints}
\usepackage{enumitem}
\usepackage{graphicx}
\usepackage{color}
\usepackage{hyperref}

\newtheorem{theorem}{Theorem}
\newtheorem{proposition}[theorem]{Proposition}
\newtheorem{lemma}[theorem]{Lemma}
\newtheorem{corollary}[theorem]{Corollary}

\newtheorem*{definition}{Definition}
\newtheorem{remark}{Remark}

\newtheorem{exam}{Example}

\newenvironment{proofref}[1]{\noindent{\bf Proof (of~#1):}}{\qed\medskip}
\renewenvironment{proof}{\noindent{\bf Proof:}}{\qed\medskip}

\renewcommand{\leq}{\leqslant}
\renewcommand{\geq}{\geqslant}
\renewcommand{\P}{\mathbb{P}}

\renewcommand{\phi}{\varphi}
\renewcommand{\epsilon}{\varepsilon}
\renewcommand{\hat}{\widehat}

\newcommand{\defeq}{\stackrel{\mbox{\scriptsize \rm def}}{=}}
\newcommand{\cA}{\mathcal{A}}
\newcommand{\cB}{\mathcal{B}}
\newcommand{\cE}{\mathcal{E}}
\newcommand{\cF}{\mathcal{F}}
\newcommand{\cG}{\mathcal{G}}
\newcommand{\cH}{\mathcal{H}}
\newcommand{\cL}{\mathcal{L}}
\newcommand{\cN}{\mathcal{N}}
\newcommand{\cX}{\mathcal{X}}
\newcommand{\E}{\mathbb{E}}

\newcommand{\Q}{\mathbb{Q}}
\newcommand{\R}{\mathbb{R}}
\newcommand{\KL}{\mathop{\mathrm{KL}}}
\newcommand{\kl}{\mathop{\mathrm{kl}}}

\newcommand{\e}{\mathrm{e}}

\newcommand{\indicator}[1]{\mathds{1}_{#1}}
\newcommand{\ind}[1]{\indicator{#1}}
\renewcommand{\d}{\,\mathrm{d}}
\newcommand{\dd}{\mathrm{d}}
\newcommand{\lm}{\mathfrak{m}}
\newcommand{\Ber}{\mathrm{Ber}}
\newcommand{\ol}{\overline}

\newcommand{\Div}{\mathrm{Div}_{\! f}}
\renewcommand{\div}{\mathrm{div}_{\! f}}

\renewcommand{\tilde}{\widetilde}

\newcommand{\norm}[1]{\left\Vert #1 \right\Vert}
\newcommand{\Pac}{\P_{\mbox{\scriptsize \rm ac}}}
\newcommand{\Psing}{\P_{\mbox{\scriptsize \rm sing}}}
\newcommand{\ac}{\mbox{\scriptsize \rm ac}}
\newcommand{\sing}{\mbox{\scriptsize \rm sing}}

\newcommand{\transp}{\mbox{\tiny T}}
\newcommand{\Rb}{R_{\mbox{\tiny \rm Bayes}}}

\makeatletter
\renewcommand*{\@seccntformat}[1]{\csname the#1\endcsname.\quad}
\makeatother

\newcommand{\papertitle}{Fano's inequality for random variables}
\title{\papertitle\footnote{The authors would like to thank Aur{\'e}lien Garivier, Jean-Baptiste Hiriart-Urruty and Vincent Tan Yan Fu for their insightful comments and suggestions. This work was partially supported by the CIMI (Centre International de Math\'{e}matiques et d'Informatique) Excellence program. The authors acknowledge  the  support of the French Agence Nationale  de la Recherche (ANR), under  grants ANR-13-BS01-0005 (project SPADRO) and ANR-13-CORD-0020 (project ALICIA). Gilles Stoltz would like to thank Investissements d'Avenir (ANR-11-IDEX-0003/Labex Ecodec/ANR-11-LABX-0047) for financial support.
}}

\author{
S{\'e}bastien Gerchinovitz \\
Pierre M{\'e}nard \vspace{.25cm} \\
IMT --- Universit{\'e} Paul Sabatier, Toulouse, France \vspace{.25cm} \\
{sebastien.gerchinovitz@math.univ-toulouse.fr} \\
{pierre.menard@math.univ-toulouse.fr} \vspace{.25cm} \\
\& \vspace{.25cm} \\
Gilles Stoltz \vspace{.25cm} \\
Laboratoire de Math{\'e}matiques d'Orsay \\
Universit{\'e} Paris-Sud, CNRS, Universit{\'e} Paris-Saclay, Orsay, France \vspace{.25cm} \\
GREGHEC --- HEC Paris, CNRS \vspace{.25cm} \\
{gilles.stoltz@math.u-psud.fr}
}

\date{\today \vspace{-.25cm}}

\begin{document}
\setlist{noitemsep,topsep=0pt,parsep=0pt,partopsep=0pt}
\renewcommand\labelitemi{--}

\pagestyle{fancy}
\renewcommand{\footrulewidth}{0.4pt}

\maketitle

\vspace*{3pt}\noindent\rule{\linewidth}{0.8pt}\vspace{5pt}
\thispagestyle{empty}

\begin{abstract}
We extend Fano's inequality,
which controls the average probability of events in terms of
the average of some $f$--divergences, to work with arbitrary
events (not necessarily forming a partition) and even with arbitrary $[0,1]$--valued
random variables, possibly in continuously infinite number.
We provide two applications of these extensions, in which the consideration
of random variables is particularly handy: we offer
new and elegant proofs for existing lower bounds, on
Bayesian posterior concentration (minimax or distribution-dependent) rates
and on the regret in non-stochastic sequential learning.
\end{abstract}

{\small
}

\newpage
\section{Introduction}

Fano's inequality is a popular information-theoretical result that provides a lower bound on worst-case error probabilities in multiple-hypotheses testing problems. It has important consequences in information theory \citep{CoTh06} and related fields. In mathematical statistics, it has become a key tool to derive lower bounds on minimax (worst-case) rates of convergence for various statistical problems such as nonparametric density estimation, regression, and classification (see, e.g., \citealp{Tsy-09-NonParametric,Massart03StFlour}).

Multiple variants of Fano's inequality have been derived in the literature. They can handle a finite, countable, or even continuously infinite number of hypotheses. Depending on the community, it has been stated in various ways. In this article, we focus on statistical versions of Fano's inequality. For instance, its most classical version states that for all sequences of $N \geq 2$ probability distributions $\P_1,\ldots,\P_N$ on the same measurable space $(\Omega,\cF)$, and all events $A_1,\ldots,A_N$ forming a partition of $\Omega$, \vspace{-.2cm}
\[
\frac{1}{N} \sum_{i=1}^N \P_i(A_i) \leq \frac{\displaystyle{\frac{1}{N} \inf_{\Q} \sum_{i=1}^N \KL(\P_i,\Q)} + \ln(2)}{\ln(N)}\,, \vspace{-.1cm}
\]
where the infimum in the right-hand side is over all probability distributions $\Q$ on $(\Omega,\cF)$.
The link to multiple-hypotheses testing is by considering events of the form $A_i = \{ \hat{\theta}=i\}$, where
$\hat{\theta}$ is an estimator of $\theta$. Lower bounds on the average of the $\P_i\bigl(\hat{\theta} \neq i\bigr)$
are then obtained.

Several extensions to more complex settings were derived in the past. For example, \citet{HV94} addressed the case of countably infinitely many probability distributions, while \citet{DuWa-13-Fanocontinuum} and \citet{ChenETAL-14-BayesRiskLowerBounds} further generalized Fano's inequality to continuously infinitely many distributions; see also~\citet{AeSaZh10}. \citet{Gus-03-Fano} extended Fano's inequality in two other directions, first by considering $[0,1]$--valued random variables $Z_i$ such that $Z_1 + \ldots + Z_N = 1$, instead of the special case $Z_i = \indicator{A_i}$, and second, by considering $f$--divergences.
All these extensions, as well as others recalled in Section~\ref{sec:ref-beg}, provide a variety of tools that adapt nicely to the variety of statistical problems.

\paragraph{Content and outline of this article.}
In this article, we first revisit and extend Fano's inequality and then provide new applications.
More precisely, Section~\ref{sec:divf} recalls the definition of $f$--divergences and states our main ingredient
for our extended Fano's inequality, namely, a data-processing inequality with expectations
of random variables. The short Section~\ref{sec:FanoExample} is a pedagogical version
of the longer Section~\ref{sec:prooftech}, where we explain and illustrate our two-step methodology
to establish new versions of Fano's inequality: a Bernoulli reduction is followed by
careful lower bounds on the $f$--divergences between two Bernoulli distributions.
In particular, we are able to extend Fano's inequality to both continuously many distributions $\P_{\theta}$ and arbitrary events $A_\theta$ that
do not necessarily form a partition or to arbitrary $[0,1]$--valued random variables $Z_{\theta}$ that are not required to sum up (or integrate) to $1$.
We also point out that the alternative distribution $\Q$ could vary with $\theta$.
We then move on in Section~\ref{sec:applis}
to our main new statistical applications, illustrating in particular that it is handy to be able to consider random variables not necessarily
summing up to~$1$. The two main such applications deal with Bayesian posterior concentration lower bounds and a regret lower
bound in non-stochastic sequential learning. (The latter application, however, could be obtained by the extension by~\citealp{Gus-03-Fano}.)
Section~\ref{sec:otherapplications} presents two other applications
which---perhaps surprisingly---follow from the special case $N=1$ in Fano's inequality.
One of these applications is about distribution-dependent lower bounds on
Bayesian posterior concentration (elaborating on results by~\citealp{HoRoSH-15-AdaptivePosteriorConcentrationRates}).
The end of the article provides a review of the literature in Section~\ref{sec:ref-beg};
it explains, in particular, that the Bernoulli reduction lying at the heart of our analysis was
already present, at various levels of clarity, in earlier works.
Finally, Section~\ref{sec:proofLBKL} provides
new and simpler proofs of some important lower bounds on the Kullback-Leibler
divergence, the main contributions being a short and enlightening
proof of the refined Pinsker's inequality by~\citet{OrWe-05-PinskerDistributionDependent},
and a sharper \citet{BrHu-78-RisqueMinimax,BrHu-79-RisqueMinimax} inequality.

\newpage
\section{Data-processing inequality with expectations of random variables}
\label{sec:divf}

This section collects the definition of and some well-known results about $f$--divergences, a special case of which is given by the
Kullback-Leibler divergence. It also states a recent and less known result, called the data-processing inequality with expectations of random variables;
it will be at the heart of the derivation of our new Fano's inequality for random variables.

\subsection{Kullback-Leibler divergence}
Let $\P,\Q$ be two probability distributions on the same measurable space $(\Omega,\cF)$. We write $\P \ll \Q$ to indicate that $\P$ is absolutely continuous with respect to $\Q$. The Kullback-Leibler divergence $\KL(\P,\Q)$ is defined by
\[
\KL(\P,\Q) = \left\{\begin{array}{ll}
	\displaystyle \bigintsss_{\Omega} \ln\!\left(\frac{\d \P}{\d \Q}\right)\! \d \P & \textrm{if $\P \ll \Q$}; \smallskip \\
	+ \infty & \textrm{otherwise}.
\end{array} \right.
\]
We write $\Ber(p)$ for the Bernoulli distribution with parameter $p$. We also use the usual measure-theoretic conventions in $\R \cup \{+\infty\}$;
in particular $0 \times (+\infty) = 0$ and $1/0 = +\infty$, as well as $0/0 = 0$. We also set $\ln(0) = -\infty$ and $0 \ln(0) = 0$.

The Kullback-Leibler divergence function $\kl$ between Bernoulli distributions equals,
for all $(p,q) \in [0,1]^2$,
\[
\kl(p,q) \defeq \KL\bigl(\Ber(p),\Ber(q) \bigr) = p \ln \!\left( \frac{p}{q} \right) + (1-p) \ln \!\left( \frac{1-p}{1-q} \right).
\]

Kullback-Leibler divergences are actually a special case of $f$--divergences with $f(x) = x \ln x$;
see \citealp{Csi-63-fdivergence}, \citealp{AlSi-66-fdivergences} and \citealp{Gus-03-Fano} for further details.

\subsection{$f$--divergences}
Let $f : (0,+\infty) \to \R$ be any convex function satisfying $f(1)=0$. By convexity, we can define \[ f(0) \defeq \lim_{t \downarrow 0} f(t) \in \R \cup \{+\infty\}\,;\] the extended function $f: [0,+\infty) \to \R \cup \{+\infty\}$ is still convex.

Before we may actually state the definition of $f$--divergences,
we recall the definition of the maximal slope $M_f$ of a convex function $f$
and provide notation for the Lebesgue decomposition of measures.

\paragraph{Maximal slope.}
For any $x > 0$, the limit
\[
\lim_{t \to +\infty} \frac{f(t)-f(x)}{t-x}
= \sup_{t > 0} \frac{f(t)-f(x)}{t-x} \in [0,+\infty]
\]
exists since (by convexity) the slope $\bigl( f(t) - f(x) \bigr)/(t-x)$ is non-decreasing as $t$ increases. Besides, this limit does not depend on $x$ and equals
\[
M_f \defeq \lim_{t \to +\infty} \frac{f(t)}{t} \in (-\infty,+\infty]\,,
\]
which thus represents the maximal slope of $f$.
A useful inequality following from the two equations above with $t=x+y$ is
\[
\forall \, x > 0, \ y > 0\,,
\qquad \frac{f(x+y) - f(x)}{y} \leq M_f \,.
\]
Put differently,
\begin{equation}
\label{eq:Mf}
\forall \, x \geq 0, \ y \geq 0, \qquad f(x+y) \leq f(x) + y\,M_f\,,
\end{equation}
where the extension to $y = 0$ is immediate and the one to $x = 0$ follows by continuity of $f$ on $(0,+\infty)$, which itself follows from its convexity.

\paragraph{Lebesgue decomposition of measures.}
We recall that $\ll$ denotes the absolute continuity between measures
and we let $\bot$ denote the fact that two measures are singular.
For distributions
$\P$ and $\Q$ defined on the same measurable space $(\Omega,\cF)$,
the Lebesgue decomposition of $\P$ with respect to~$\Q$ is denoted by
\begin{equation}
\label{eq:Lbgdec}
\P = \Pac + \Psing\,, \qquad \mbox{where} \qquad
\Pac \ll \Q \quad \mbox{and} \quad \Psing \bot \Q \,,
\end{equation}
so that $\Pac$ and $\Psing$ are both sub-probabilities (positive measures
with total mass smaller than or equal to $1$) and, by definition,
\[
\frac{\d \P}{\d \Q} = \frac{\d \Pac}{\d \Q}\,.
\]

\paragraph{Definition of $f$--divergences.}
The existence of the integral in the right-hand side of
the definition below follows from
the general form of Jensen's inequality stated in Lemma~\ref{lem:Jensen2} (Appendix~\ref{sec:Jensen}) with $\phi=f$ and $C=[0,+\infty)$.

\begin{definition}
\label{def:fdiv}
Given a convex function $f : (0,+\infty) \to \R$ satisfying $f(1)=0$,
the $f$--divergence $\Div(\P,\Q)$ between two probability distributions
on the same measurable space $(\Omega,\cF)$ is defined as
\begin{equation}
\Div(\P,\Q) =
\bigintsss_{\Omega} f\!\left(\frac{\d \P}{\d \Q}\right)\! \d \Q
+ \Psing(\Omega) \, M_f\,.
\end{equation}
\end{definition}

Jensen's inequality of Lemma~\ref{lem:Jensen2}, together with~\eqref{eq:Mf}, also indicates that
$\Div(\P,\Q) \geq 0$. Indeed,
\[
\bigintsss_{\Omega} f\!\left(\frac{\d \P}{\d \Q}\right)\! \d \Q
\geq f\!\left(\int_{\Omega} \frac{\d \P}{\d \Q} \d \Q \right)
= f\bigl( \Pac(\Omega) \bigr)\,,
\]
so that by~\eqref{eq:Mf},
\[
\Div(\P,\Q) \geq f\bigl( \Pac(\Omega) \bigr) +
\Psing(\Omega) \, M_f \geq f\bigl( \Pac(\Omega) + \Psing(\Omega) \bigr) = f(1) = 0\,. \\
\]
\medskip

Concrete and important examples of $f$--divergences, such as
the Hellinger distance and the $\chi^2$--divergence,
are discussed in details in Section~\ref{sec:prooftech}. The Kullback-Leibler divergence corresponds to $\Div$ with the function $f : x \mapsto x \, \ln(x)$.
We have
$M_f = +\infty$ for the Kullback-Leibler and $\chi^2$--divergences,
while $M_f = 1$ for the Hellinger distance.

\subsection{The data-processing inequality and two major consequences}
\label{sec:DPfdiv}

The data-processing inequality (also called contraction of relative entropy in
the case of the Kullback-Leibler divergence)
indicates that transforming the data at hand can only reduce the ability to distinguish between two probability distributions.

\begin{lemma}[Data-processing inequality]
\label{lem:dataProcessingIneq--fdiv}
Let $\P$ and $\Q$ be two probability distributions on the same measurable space $(\Omega,\cF)$, and let $X$ be any random variable on $(\Omega,\cF)$. Denote by $\P^X$ and $\Q^X$ the associated pushforward measures (the laws of $X$ under $\P$ and $\Q$). Then,
\[
\Div\!\left(\P^X,\Q^X\right) \leq \Div(\P,\Q)\,.
\]
\end{lemma}

\begin{corollary}[Data-processing inequality with expectations of random variables]
\label{lem:dataProcessingIneq2--fdiv}
Let $\P$ and $\Q$ be two probability distributions on the same measurable space $(\Omega,\cF)$, and let $X$ be any random variable on $(\Omega,\cF)$ taking values
in $[0,1]$. Denote by $\E_{\P}[X]$ and $\E_{\Q}[X]$ the expectations of $X$ under $\P$ and $\Q$ respectively. Then,
\[
\div\bigl(\E_{\P}[X],\E_{\Q}[X]\bigr) \leq \Div(\P,\Q)\,,
\]
where $\div(p,q) = \Div\bigl(\Ber(p),\Ber(q)\bigr)$
denotes the $f$--divergence between Bernoulli distributions
with respective parameters $p$ and $q$.
\end{corollary}

\begin{corollary}[Joint convexity of $\Div$]
\label{cor:jointconvKL--fdiv}
All $f$--divergences $\Div$ are jointly convex, i.e.,
for all probability distributions $\P_1,\P_2$ and $\Q_1,\Q_2$ on the same measurable space $(\Omega,\cF)$,
and all $\lambda \in (0,1)$,
\[
\Div \Bigl( (1-\lambda) \P_1 + \lambda \P_2, \, (1-\lambda) \Q_1 + \lambda \Q_2 \Bigr) \leq
(1-\lambda) \, \Div(\P_1,\Q_1) + \lambda \, \Div(\P_2,\Q_2)\,.
\]
\end{corollary}

Lemma~\ref{lem:dataProcessingIneq--fdiv} and~Corollary~\ref{cor:jointconvKL--fdiv}
are folklore knowledge. However, for the sake of self-completeness,
we provide complete and elementary proofs thereof
in the extended version of this article (see Appendix~\ref{sec:Ext}).
The proof of Lemma~\ref{lem:dataProcessingIneq--fdiv} is extracted from \citet[Section 4.2]{AlSi-66-fdivergences}, see
also \citet[Proposition~1.2]{ThInfo-Pardo},
while we derive Corollary~\ref{cor:jointconvKL--fdiv} as an elementary consequence of
Lemma~\ref{lem:dataProcessingIneq--fdiv} applied to an augmented probability space.
These proof techniques do not seem to be well known;
indeed, in the literature
many proofs of the elementary properties above for the Kullback-Leibler divergence
focus on the discrete case (\citealp{CoTh06}) or use the duality formula for the Kullback-Leibler divergence
(\citealp{Massart03StFlour} or \citealp{BoLuMa12}, in particular Exercise 4.10 therein).

On the contrary, Corollary~\ref{lem:dataProcessingIneq2--fdiv} is a recent though elementary result, proved
in~\citet{GaMeSt16} for Kullback-Leibler divergences.
The proof readily extends to $f$--divergences. \\

\begin{proofref}{Corollary~\ref{lem:dataProcessingIneq2--fdiv}}
We augment the underlying measurable space into $\Omega \times [0,1]$, where
$[0,1]$ is equipped with the Borel $\sigma$--algebra $\cB\bigl([0,1]\bigr)$ and the Lebesgue measure $\lm$.
We denote by $\P \otimes \lm$ and $\Q \otimes \lm$ the product distributions
of $\P$ and $\lm$, $\Q$ and $\lm$.
We write the Lebesgue decomposition $\P = \Pac + \Psing$ of $\P$ with respect to $\Q$,
and deduce from it the Lebesgue decomposition of
$\P \otimes \lm$ with respect to $\Q \otimes \lm$:
the absolutely continuous part is given by $\Pac \otimes \lm$, with
density
\[
(\omega,x) \in \Omega \times [0,1] \longmapsto
\frac{\d(\Pac \otimes \lm)}{\d(\Q \otimes \lm)}(\omega,x)
= \frac{\d\Pac}{\d\Q}(\omega)\,,
\]
while the singular part is given by $\Psing \otimes \lm$, a subprobability
with total mass $\Psing(\Omega)$.
In particular,
\[
\Div\bigl(\P \otimes \lm,\,\Q \otimes \lm \bigr)
= \Div(\P,\Q)\,.
\]
Now, for all events $E \in \cF \otimes \cB\bigl([0,1]\bigr)$,
the data-processing inequality (Lemma~\ref{lem:dataProcessingIneq--fdiv})
used with the indicator function $X = \indicator{E}$
ensures that
\[
\Div\bigl(\P \otimes \lm,\,\Q \otimes \lm \bigr)
\geq \Div\Bigl( (\P \otimes \lm)^{\ind{E}},\,\,(\Q \otimes \lm)^{\ind{E}} \Bigr)
= \div \bigl( (\P \otimes \lm)(E), \, (\Q \otimes \lm)(E) \bigr)\,,
\]
where the final equality is by mere definition
of $\div$ as the $f$--divergence between Bernoulli distributions.
The proof is concluded by noting that for the choice of $E = \bigl\{(\omega,x) \in \Omega \times [0,1]: \, x \leq X(\omega)\bigr\}$, Tonelli's theorem ensures that
\[
(\P \otimes \lm)(E) = \bigintsss_{\Omega} \left( \int_{[0,1]} \indicator{\bigl\{x \leq X(\omega)\bigr\}} \d \mathfrak{m}(x) \right) \dd\P(\omega) = \E_{\P}[X] \,,
\]
and, similarly, $(\Q \otimes \lm)(E) = \E_{\Q}[X]$.
\end{proofref}

\newpage
\section{How to derive a Fano-type inequality: an example}
\label{sec:FanoExample}

In this section we explain on an example the methodology to derive Fano-type inequalities. We will
present the generalization of the approach and the resulting bounds in Section~\ref{sec:prooftech}, but the proof below already contains the two key arguments: a reduction to Bernoulli distributions, and a lower bound on the $f$--divergence between Bernoulli distributions. For the sake of concreteness,
we focus on the Kullback-Leibler divergence in this section.
We recall that we will discuss how novel (or not novel) our results and approaches are in Section~\ref{sec:ref-beg}.

\begin{proposition}
\label{prop:FanoIT}
Given an underlying measurable space,
for all probability pairs $\P_i,\,\Q_i$ and all events $A_i$ (not necessarily disjoint), where $i \in \{1,\ldots,N\}$, with $0 < \frac{1}{N} \sum_{i=1}^N \Q_i(A_i) < 1$, we have
\[
\frac{1}{N} \sum_{i=1}^N \P_i(A_i) \leq \frac{\displaystyle{\frac{1}{N} \sum_{i=1}^N \KL(\P_i,\Q_i)} + \ln(2)}{\displaystyle{
- \ln \Biggl( \frac{1}{N} \sum_{i=1}^N \Q_i(A_i) \Biggr)}}\,.
\]
In particular, if $N \geq 2$ and the $A_i$ form a partition,
\[
\frac{1}{N} \sum_{i=1}^N \P_i(A_i) \leq \frac{\displaystyle{\frac{1}{N} \inf_{\Q} \sum_{i=1}^N \KL(\P_i,\Q)} + \ln(2)}{\ln(N)}\,.
\]
\end{proposition}

\begin{proof}
Our first step is to reduce the problem to Bernoulli distributions. Using first the joint convexity of the Kullback-Leibler divergence (Corollary~\ref{cor:jointconvKL--fdiv}), and second the data-processing inequality with the indicator functions $X = \indicator{A_i}$ (Lemma~\ref{lem:dataProcessingIneq--fdiv}), we get
\begin{equation}
\label{eq:red1}
\kl\Biggl(\frac{1}{N} \sum_{i=1}^N \P_i(A_i),\,\frac{1}{N} \sum_{i=1}^N \Q_i(A_i)\Biggr) \leq
\frac{1}{N} \sum_{i=1}^N \kl\bigl( \P_i(A_i), \Q_i(A_i) \bigr) \leq
\frac{1}{N} \sum_{i=1}^N \KL(\P_i,\Q_i)\,.
\end{equation}
Therefore, we have $\kl\bigl(\ol{p},\ol{q}\bigr) \leq \ol{K}$ with
\begin{equation}
\ol{p} = \frac{1}{N} \sum_{i=1}^N \P_i(A_i) \qquad \ol{q} = \frac{1}{N} \sum_{i=1}^N \Q_i(A_i) \qquad \ol{K} = \frac{1}{N} \sum_{i=1}^N \KL(\P_i,\Q_i) \,.
\label{eq:defAverageQuantities}
\end{equation}
\noindent
Our second and last step is to lower bound  $\kl\bigl(\ol{p},\ol{q}\bigr)$ to extract an upper bound on $\ol{p}$. Noting that $\ol{p} \ln\bigl(\ol{p}\bigr) + \bigl(1-\ol{p}\bigr) \ln\bigl(1-\ol{p}\bigr) \geq - \ln(2)$, we have, by definition of $\kl\bigl(\ol{p},\ol{q}\bigr)$,
\begin{equation}
\label{eq:bd1-intro}
\kl\bigl(\ol{p},\ol{q}\bigr) \geq \ol{p} \ln\bigl(1/\ol{q}\bigr) - \ln(2)\,,
\qquad \mbox{thus} \qquad
\ol{p} \leq \frac{\kl\bigl(\ol{p},\ol{q}\bigr) + \ln(2)}{\ln\bigl(1/\ol{q}\bigr)}
\end{equation}
where $\ol{q} \in (0,1)$ by assumption.
Substituting the upper bound $\kl\bigl(\ol{p},\ol{q}\bigr) \leq \ol{K}$ in~\eqref{eq:bd1-intro} concludes the proof.
\end{proof}

\newpage
\section{Various Fano-type inequalities, with the same two ingredients}
\label{sec:prooftech}

We extend the approach of Section~\ref{sec:FanoExample} and derive a broad family of Fano-type inequalities, which will be of the form
\[
\ol{p} \leq \psi\bigl(\ol{q},\ol{K}\bigr)\,,
\]
where the average quantities $\ol{p}$, $\ol{q}$ and $\ol{K}$
are described in Section~\ref{sec:avg} (first ingredient)
and where the functions $\psi$ are described in Section~\ref{sec:LBdiv}
(second ingredient). The simplest example that we considered in Section~\ref{sec:FanoExample} corresponds to $\psi(q,K) = \bigl(K + \ln(2)\bigr)/\ln(1/q)$ and
\[
\ol{p} = \frac{1}{N} \sum_{i=1}^N \P_i(A_i) \qquad \ol{q} = \frac{1}{N} \sum_{i=1}^N \Q_i(A_i) \qquad \ol{K} = \frac{1}{N} \sum_{i=1}^N \KL(\P_i,\Q_i) \,.
\]
We address here the more general cases where the finite averages are replaced with integrals over any measurable space $\Theta$
and where the indicator functions $\indicator{A_i}$ are replaced with arbitrary $[0,1]$--valued random variables $Z_{\theta}$, where $\theta \in \Theta$.

We recall that the novelty (or lack of novelty) of our results will be discussed in detail
in Section~\ref{sec:ref-beg}; of particular interest therein is the discussion of the (lack of) novelty
of our first ingredient, namely the reduction to Bernoulli distributions.

\subsection{Reduction to Bernoulli distributions}
\label{sec:avg}

As in Section~\ref{sec:FanoExample}, we can resort to the
data-processing inequality (Lemma~\ref{lem:dataProcessingIneq--fdiv})
to lower bound any $f$--divergence by that of suitably chosen Bernoulli distributions.
We present three such reductions, in increasing degree of generality. We only indicate
how to prove the first one, since they are all similar.

\paragraph{Countably many distributions.}
We consider some underlying measurable space,
countably many pairs of probability distributions $\P_i,\,\Q_i$ on this space,
not necessarily disjoint events $A_i$, all indexed by $i \in \{1,2,\ldots\}$,
as well as a convex combination $\alpha = (\alpha_1,\alpha_2,\ldots)$.
The latter can be thought of as a prior distribution.
The inequality reads
\begin{equation}
\label{eq:red2}
\div\Biggl(\sum_{i \geq 1} \alpha_i \, \P_i(A_i), \, \sum_{i \geq 1} \alpha_i \, \Q_i(A_i)\Biggr) \leq
\sum_{i \geq 1} \alpha_i \div\bigl( \P_i(A_i), \Q_i(A_i) \bigr) \leq
\sum_{i \geq 1} \alpha_i \Div(\P_i,\Q_i)\,.
\end{equation}
The second inequality of~\eqref{eq:red2} follows from the data-processing inequality
(Lemma~\ref{lem:dataProcessingIneq--fdiv}) by considering the indicator functions $X = \indicator{A_i}$.
For the first inequality, we resort to a general version of Jensen's inequality stated in Lemma~\ref{lem:Jensen2} (Appendix~\ref{sec:Jensen}),
by considering the convex function $\phi = \div$ (Corollary~\ref{cor:jointconvKL--fdiv}) on the convex set $C=[0,1]^2$,
together with the probability measure
\[
\mu = \sum_i \alpha_i \, \delta_{\!\bigl(\P_i(A_i), \Q_i(A_i)\bigr)}\,,
\]
where $\delta_{(x,y)}$ denotes the Dirac mass at $(x,y) \in \R^2$.

\paragraph{Distributions indexed by a possibly continuous set.}
Up to measurability issues (that are absent in the countable case),
the reduction above immediately extends to the case of statistical models $\P_\theta,\,\Q_\theta$
and not necessarily disjoint events $A_\theta$
indexed by a measurable parameter space $(\Theta,\mathcal{G})$, equipped with
a prior probability distribution $\nu$ on~$\Theta$.
We assume that
\[
\theta \in \Theta \longmapsto \bigl( \P_\theta(A_\theta), \, \Q_\theta(A_\theta) \bigr) \qquad \textrm{and} \qquad \theta \in \Theta \longmapsto \Div(\P_\theta,\Q_\theta)
\]
are $\mathcal{G}$--measurable and get the reduction
\begin{equation}
\label{eq:red3}
\div\Biggl( \int_\Theta \P_\theta(A_\theta) \d\nu(\theta), \, \int_\Theta \Q_\theta(A_\theta) \d\nu(\theta) \Biggr)
\leq \int_\Theta \div\bigl( \P_\theta(A_\theta), \, \Q_\theta(A_\theta) \bigr) \d\nu(\theta)
\leq \int_\Theta \Div(\P_\theta,\Q_\theta)\d\nu(\theta)\,.
\end{equation}

\paragraph{Random variables.}
In the reduction above, it was unnecessary that the sets $A_\theta$ form a partition or even be disjoint.
It is therefore not surprising that it can be generalized by replacing the indicator functions $\indicator{A_{\theta}}$
with arbitrary $[0,1]$--valued random variables $Z_{\theta}$.
We denote the expectations of the latter with respect to $\P_\theta$ and $\Q_\theta$
by $\E_{\P_\theta}$ and $\E_{\Q_\theta}$ and
assume that
\[
\theta \in \Theta \longmapsto \Bigl( \E_{\P_\theta} \bigl[ Z_\theta \bigr],
\,\, \E_{\Q_\theta} \bigl[ Z_\theta \bigr] \Bigr) \qquad \textrm{and} \qquad \theta \in \Theta \longmapsto \Div(\P_\theta,\Q_\theta)
\]
are $\mathcal{G}$--measurable.
The reduction reads in this case
\begin{align}
\nonumber
\div\Biggl(\int_\Theta \E_{\P_\theta} \bigl[ Z_\theta \bigr] \d\nu(\theta), \,
\int_\Theta \E_{\Q_\theta} \bigl[ Z_\theta \bigr] \d\nu(\theta) \Biggr)
& \leq \bigintsss_\Theta \div\Bigl( \E_{\P_\theta} \bigl[ Z_\theta \bigr], \, \E_{\Q_\theta} \bigl[ Z_\theta \bigr] \Bigr) \d\nu(\theta) \\
\label{eq:red5}
& \leq \int_\Theta \Div(\P_\theta,\Q_\theta)\d\nu(\theta)\,,
\end{align}
where the first inequality relies on convexity of $\div$ and on Jensen's inequality,
and the second inequality follows from the data-processing inequality with expectations of random variables
(Lemma~\ref{lem:dataProcessingIneq2--fdiv}).

\subsection{Any lower bound on $\div$ leads to a Fano-type inequality}
\label{sec:LBdiv}

The section above indicates that after the reduction to the Bernoulli case,
we get inequations of the form ($\ol{p}$ is usually the unknown)
\[
\div\bigl(\ol{p},\ol{q}\bigr) \leq \ol{D}\,,
\]
where $\ol{D}$ is an average of $f$--divergences, and
$\ol{p}$ and $\ol{q}$ are averages of probabilities of events or expectations
of $[0,1]$--valued random variables.
We thus proceed by lower bounding the $\div$ function.
The lower bounds are idiosyncratic to each $f$--divergence
and we start with the most important one, namely, the Kullback-Leibler divergence.

\paragraph{Lower bounds on $\kl$.}
The most classical bound
was already used in Section~\ref{sec:FanoExample}:
for all $p \in [0,1]$ and $q \in (0,1)$,
\begin{equation}
\label{eq:bd1}
\kl(p,q) \geq p \ln(1/q) - \ln(2)\,,
\qquad \mbox{thus} \qquad
p \leq \frac{\kl(p,q) + \ln(2)}{\ln(1/q)}\,.
\end{equation}
It is well-known that this bound can be improved by replacing the term $\ln(2)$ with $\ln(2-q)$:
for all $p \in [0,1]$ and $q \in (0,1)$,
\begin{equation}
\label{eq:bd2}
\kl(p,q) \geq p \ln(1/q) - \ln(2-q)\,,
\qquad \mbox{thus} \qquad
p \leq \frac{\kl(p,q) + \ln(2-q)}{\ln(1/q)}\,.
\end{equation}
This leads to a non-trivial bound even if $q=1/2$ (as is the case in some applications).
A (novel) consequence of this bound is that
\begin{equation}
\label{eq:bd2bis}
p \leq 0.21 + 0.79\,q + \frac{\kl(p,q)}{\ln(1/q)}\,.
\end{equation}
The improvement~\eqref{eq:bd2} is a consequence of, e.g., a convexity inequality, and its
proof and the one for ~\eqref{eq:bd2bis} can be found in Section~\ref{sec:proofLBKL-conv}.

The next and final bound makes a connection between Pinsker's and Fano's inequalities:
on the one hand, it is a refined Pinsker's inequality and on the other hand,
it leads to a bound on $p$ of the same flavor as \eqref{eq:bd1}--\eqref{eq:bd2bis}.
Namely, for all $p \in [0,1]$ and $q \in (0,1)$,
\begin{equation}
\label{eq:bd3}
\kl(p,q) \geq \max\!\left\{\ln\!\left(\frac{1}{q}\right)\!,\,2\right\} (p-q)^2\,,
\qquad \mbox{thus} \qquad
p \leq q + \sqrt{\frac{\kl(p,q)}{\max\bigl\{\ln(1/q),2\bigr\}}}\,.
\end{equation}
The first inequality was stated and proved by \citet{OrWe-05-PinskerDistributionDependent},
the second is a novel but straightforward consequence of it.
We provide their proofs and additional references in Section~\ref{sec:proofLBKL-Pinsker}.

\paragraph{Lower bound on $\div$ for the $\chi^2$ divergence.}
This case corresponds to $f(x) = x^2-1$.
The associated divergence equals $+\infty$ when $\P \not\ll \Q$, and when $\P \ll \Q$,
\[
\chi^2(\P,\Q) = \bigintsss_{\Omega} \left(\frac{\d \P}{\d \Q}\right)^{\!\! 2} \! \d \Q - 1\,.
\]
A direct calculation and the usual measure-theoretic conventions entail
the following simple lower bound: for all $(p,q) \in [0,1]^2$,
\begin{equation}
\label{eq:chi2LB}
\chi^2\bigl(\Ber(p),\Ber(q)\bigr) = \frac{(p-q)^2}{q(1-q)} \geq \frac{(p-q)^2}{q}\,,
\qquad \mbox{thus} \qquad
p \leq q + \sqrt{q \, \chi^2\bigl(\Ber(p),\Ber(q)\bigr)}\,.
\end{equation}

\paragraph{Lower bound on $\div$ for the Hellinger distance.}
This case corresponds to $f(x) = \bigl( \sqrt{x} - 1 \bigr)^2$, for which $M_f = 1$.
The associated divergence equals, when $\P \ll \Q$,
\[
H^2(\P,\Q) = \bigintss_{\Omega} \left( \sqrt{\frac{\d \P}{\d \Q}} - 1 \right)^{\!\! 2} \! \d \Q
= 2 \! \left(  1 -  \bigintss_{\Omega} \sqrt{\frac{\d \P}{\d \Q}} \,\d \Q \right)
\]
and always lies in $[0,2]$.
A direct calculation indicates that
for all $p \in [0,1]$ and $q \in (0,1)$,
\[
h^2(p,q) \defeq H^2\bigl(\Ber(p),\Ber(q)\bigr)
= 2 \biggl( 1 - \Bigl( \sqrt{pq} + \sqrt{(1-p)(1-q)} \Bigr) \biggr)\,,
\]
and further direct calculations in the cases $q =0$ and $q = 1$ show
that this formula remains valid in these cases.
To get a lower bound on $h^2(p,q)$, we proceed as follows.
The Cauchy-Schwarz inequality indicates that
\[
\sqrt{pq} + \sqrt{(1-q)(1-p)}
\leq \sqrt{\bigl( p + (1-q) \bigr) \bigl( q + (1-p)\bigr)}
= \sqrt{1 - (p-q)^2}\,,
\]
or put differently, that $h^2(p,q) \geq 2 \Bigl( 1 - \sqrt{1 - (p-q)^2} \Bigr)$, thus
\begin{equation}
\label{eq:LB-H}
p \leq q + \sqrt{1-\bigl(1-h^2(p,q)/2\bigr)^2}
= q + \sqrt{h^2(p,q)\bigl(1-h^2(p,q)/4\bigr)}\,,
\end{equation}
which is one of Le Cam's inequalities.
A slightly sharper but less readable bound was
exhibited by \citet[Example~II.6]{guntuboyina2011lower}
and is provided, for the sake of completeness, in an extended version of this article, see Appendix~\ref{sec:Ext}.

\subsection{Examples of combinations}
\label{sec:excombi}

The combination of~\eqref{eq:red3} and~\eqref{eq:bd1} yields a continuous
version of Fano's inequality. (We discard again all measurability issues.)
\begin{lemma}
\label{lm:FanoContinuous}
We consider a mesurable space $(\Theta,\cE)$ equipped with a probability distribution $\nu$.
Given an underlying measurable space $(\Omega, \cF)$, for all two collections $\P_\theta,\,\Q_\theta$,
of probability distributions on this space and all collections of events $A_\theta$ of $(\Omega, \cF)$,
where $\theta \in \Theta$, with
\[
0 < \bigintsss_{\Theta} \Q_{\theta}(A_{\theta}) \d\nu(\theta) < 1\,,
\]
we have
\[
\bigintsss_{\Theta} \P_{\theta}(A_{\theta}) \d\nu(\theta) \leq
\frac{\displaystyle{\bigintsss_{\Theta} \KL \bigl(\P_{\theta}, \Q_{\theta} \bigr)} \d\nu(\theta) + \ln(2)}{- \ln \! \left(
\displaystyle{{\displaystyle{\bigintsss_{\Theta} \Q_{\theta}(A_{\theta}) \d\nu(\theta)}}} \right)}\;.
\]
\end{lemma}

The combination of~\eqref{eq:red3}, used with a uniform distribution $\nu$
on $N$ points, and~\eqref{eq:bd3} ensures the following
Fano-type inequality for finitely many random variables, whose sum does not need
to be~$1$. It will be used in our second application, in Section~\ref{sec:usecase-sequential}.
\begin{lemma}
\label{lm:Fano-suitesind}
Given an underlying measurable space,
for all probability pairs $\P_i,\,\Q_i$ and for all $[0,1]$--valued random variables
$Z_i$ defined on this measurable space, where $i \in \{1,\ldots,N\}$, with
\[
0 < \frac{1}{N} \sum_{i=1}^N \E_{\Q_i}\bigl[Z_i\bigr] < 1\,,
\]
we have
\[
\frac{1}{N} \sum_{i=1}^N \E_{\P_i}\bigl[Z_i\bigr] \leq \frac{1}{N} \sum_{i=1}^N \E_{\Q_i}\bigl[Z_i\bigr] +
\sqrt{\frac{\displaystyle{\frac{1}{N} \sum_{i=1}^N \KL(\P_i,\Q_i)}}{\displaystyle{
- \ln \Biggl( \frac{1}{N} \sum_{i=1}^N \E_{\Q_i}\bigl[Z_i\bigr] \Biggr)}}}\;.
\]
In particular, if $N \geq 2$ and $Z_1 + \ldots + Z_N = 1$ a.s., then
\[
\frac{1}{N} \sum_{i=1}^N \E_{\P_i}\bigl[Z_i\bigr] \leq \frac{1}{N} +
\sqrt{\frac{\displaystyle{\frac{1}{N} \inf_{\Q} \sum_{i=1}^N \KL(\P_i,\Q)}}{\ln(N)}}\;.
\]
\end{lemma}

For the $\chi^2$--divergence now,
the combination of, e.g., \eqref{eq:red2} in the finite and uniform case and~\eqref{eq:chi2LB}
leads to the following inequality.

\begin{lemma}
\label{lm:chi2}
Given an underlying measurable space,
for all probability pairs $\P_i,\,\Q_i$ and all events $A_i$ (not necessarily disjoint),
where $i \in \{1,\ldots,N\}$, with $0 < \frac{1}{N} \sum_{i=1}^N \Q_i(A_i) < 1$, we have
\[
\frac{1}{N} \sum_{i=1}^N \P_i(A_i) \leq
\frac{1}{N} \sum_{i=1}^N \Q_i(A_i) +
\sqrt{\frac{1}{N} \sum_{i=1}^N \Q_i(A_i)}
\,
\sqrt{\frac{1}{N} \sum_{i=1}^N \chi^2 \bigl(\P_i, \Q_i \bigr)}\,.
\]
In particular, if $N \geq 2$ and the $A_i$ form a partition,
\[
\frac{1}{N} \sum_{i=1}^N \P_i(A_i) \leq \frac{1}{N} +
\frac{1}{\sqrt{N}}
\,
\sqrt{\frac{1}{N} \inf_{\Q} \sum_{i=1}^N \chi^2 \bigl(\P_i, \Q \bigr)}\,.
\]
\end{lemma}

Similarly, for the Hellinger distance,
the simplest reduction \eqref{eq:red2} in the finite and uniform case
together with the lower bound~\eqref{eq:LB-H} yields
the following bound.

\begin{lemma}
\label{lm:Hell}
Given an underlying measurable space,
for all probability pairs $\P_i,\,\Q_i$ and all events $A_i$ (not necessarily disjoint),
where $i \in \{1,\ldots,N\}$, with $0 < \frac{1}{N} \sum_{i=1}^N \Q_i(A_i) < 1$, we have
\[
\frac{1}{N} \sum_{i=1}^N \P_i(A_i) \leq
\frac{1}{N} \sum_{i=1}^N \Q_i(A_i) +
\sqrt{\frac{1}{N} \sum_{i=1}^N H^2 \bigl(\P_i, \Q_i \bigr)}
\sqrt{1 - \frac{1}{4N} \sum_{i=1}^N H^2 \bigl(\P_i, \Q_i \bigr)}\,.
\]
In particular, if $N \geq 2$ and the $A_i$ form a partition,
\[
\frac{1}{N} \sum_{i=1}^N \P_i(A_i) \leq \frac{1}{N} +
\inf_{\Q}  \sqrt{\frac{1}{N} \sum_{i=1}^N H^2 \bigl(\P_i, \Q \bigr)}
\sqrt{1 - \frac{1}{4N} \sum_{i=1}^N H^2 \bigl(\P_i, \Q \bigr)}
\leq \frac{1}{N} +
\inf_{\Q}  \sqrt{\frac{1}{N} \sum_{i=1}^N H^2 \bigl(\P_i, \Q \bigr)}\,.
\]
\end{lemma}

\subsection{Comments on these bounds}

Section~\ref{sec:sharpness} in Appendix
discusses the sharpness of the bounds obtained above,
for the case of the Kullback-Leibler divergence.

Section~\ref{sec:Ext} provides a pointer to an extended version of this
article where the choice of a good constant alternative
distribution $\Q$ is studied. The examples of bounds derived
in Section~\ref{sec:excombi} show indeed that when
the $A_i$ form a partition, the upper bounds
feature an average $f$--divergence of the form
\[
\frac{1}{N} \inf_{\Q} \sum_{i=1}^N \Div(\P_i,\Q)
\]
and one may indeed wonder what $\Q$ should be chosen and
what bound can be achieved. Section~\ref{sec:Ext} points to
a discussion of these matters.

\newpage
\section{Main applications}
\label{sec:applis}

We present two new applications of Fano's inequality, with $[0,1]$--valued random variables $Z_i$ or $Z_{\theta}$.
The topics covered are:
\begin{itemize}
\item[--] Bayesian posterior concentration rates;
\item[--] robust sequential learning (prediction of individual sequences) in the
case of sparse losses.
\end{itemize}
As can be seen below, the fact that we are now able to consider arbitrary $[0,1]$--valued random variables $Z_{\theta}$ on a continuous parameter space $\Theta$ makes the proof of the Bayesian posterior concentration lower bound quite simple.

Two more applications will also be presented in Section~\ref{sec:otherapplications};
they have a different technical flavor, as they rely on only one pair of distributions,
i.e., $N=1$.

\subsection{Lower bounds on Bayesian posterior concentration rates}
\label{sec:bayesianPosteriorRates}

In the next paragraphs we show how our continuous Fano's inequality can be used in a simple fashion to derive lower bounds for posterior concentration rates.

\paragraph{Setting and Bayesian terminology.} We consider the following density estimation setting: we observe a sample of independent and identically distributed random variables $X_{1:n} = (X_1,\ldots,X_n)$ drawn from a probability distribution $P_{\theta}$ on $(\cX,\cF)$, with a fixed but unknown $\theta \in \Theta$. We assume that the measurable parameter space $(\Theta,\mathcal{G})$ is equipped with a prior distribution $\pi$ and that all $P_{\theta'}$ have a density $p_{\theta'}$ with respect to some reference measure $\mathfrak{m}$ on $(\cX,\cF)$.
We also assume that $(x,\theta') \mapsto p_{\theta'}(x)$ is $\cF \otimes \mathcal{G}$--measurable.
We can thus consider the transition kernel $(x_{1:n}, A) \mapsto \P_{\pi}(A\,|\,x_{1:n})$ defined for all $x_{1:n} \in \cX^n$ and all sets $A \in \mathcal{G}$ by
\begin{equation}
\displaystyle \P_{\pi}(A\,|\,x_{1:n}) = \frac{\displaystyle{\bigintsss_{\! A} \, \prod_{i=1}^n p_{\theta'}(x_i) \d\pi(\theta')}}{\displaystyle{\bigintsss_{\Theta} \, \prod_{i=1}^n p_{\theta'}(x_i) \d\pi(\theta')}}
\label{eq:def-posterior}
\end{equation}
if the denominator lies in $(0,+\infty)$; if it is null or infinite, we set, e.g., $\P_{\pi}(A\,|\,x_{1:n}) = \pi(A)$.
The resulting random measure $\P_{\pi}(\,\cdot\,\,|\,X_{1:n})$ is known as the \emph{posterior} distribution.
\smallskip

Let $\ell:\Theta \times \Theta \to \R_+$ be a measurable loss function that we assume to be a pseudo-metric\footnote{The only difference with a metric is that we allow $\ell(\theta,\theta')=0$ for $\theta \neq \theta'$.}. A posterior concentration rate with respect to $\ell$ is a sequence $(\epsilon_n)_{n \geq 1}$ of positive real numbers such that, for all $\theta \in \Theta$,
\[
\E_{\theta}\Bigl[\P_{\pi}\bigl(\theta': \ell(\theta',\theta) \leq \epsilon_n \, \big| \, X_{1:n} \bigr)\Bigr] \longrightarrow 1 \qquad \textrm{as } n \to +\infty\,,
\]
where $\E_{\theta}$ denotes the expectation with respect to $X_{1:n}$ where each $X_j$ has the $P_\theta$ law. The above convergence guarantee means that, as the size $n$ of the sample increases, the posterior mass concentrates in expectation on an $\epsilon_n$--neighborhood of the true parameter $\theta$. Several variants of this definition exist (e.g., convergence in probability or almost surely; or $\epsilon_n$ that may depend on $\theta$). Though most of these definitions can be handled with the techniques provided below, we only consider this one for the sake of conciseness.

\paragraph{Minimax posterior concentration rate.}
As our sequence $(\epsilon_n)_{n \geq 1}$ does not depend on the specific $\theta \in \Theta$ at hand,
we may study uniform posterior concentration rates: sequences $(\epsilon_n)_{n \geq 1}$ such that
\begin{equation}
\label{eq:defminimaxposteriorconcentration}
\inf_{\theta \in \Theta} \, \E_{\theta}\Bigl[\P_{\pi}\bigl(\theta': \ell(\theta',\theta) \leq \epsilon_n \, \big| \, X_{1:n} \bigr)\Bigr] \longrightarrow 1 \qquad \textrm{as } n \to +\infty\,.
\end{equation}
The minimax posterior concentration rate is given by a sequence $(\epsilon_n)_{n \geq 1}$ such that~\eqref{eq:defminimaxposteriorconcentration} holds for some prior $\pi$ while there exists a constant $\gamma \in (0,1)$
such that for all priors $\pi'$ on $\Theta$,
\[
\limsup_{n \to +\infty} \,\, \inf_{\theta \in \Theta} \, \E_{\theta}\Bigl[\P_{\pi'}\bigl(\theta': \ell(\theta',\theta) \leq \gamma \, \epsilon_n \, \big| \, X_{1:n} \bigr)\Bigr] < 1\,.
\]
We focus on proving the latter statement and provide a general technique to do so.

\begin{proposition}[A posterior concentration lower bound in the finite-dimensional Gaussian model]
\label{prop:lowerbound-posterior-minimax}
\ \\
Let $d \geq 1$ be the ambient dimension, $n \geq 1$ the sample size, and $\sigma>0$ the standard deviation. Assume we observe an $n$--sample $X_{1:n} = (X_1,\ldots,X_n)$ distributed according to $\cN(\theta,\sigma^2 I_d)$ for some unknown $\theta \in \R^d$. Let $\pi'$ be any prior distribution on $\R^d$. Then the posterior distribution $\P_{\pi'}(\,\cdot\,\,|\,X_{1:n})$ defined in~\eqref{eq:def-posterior} satisfies, for the Euclidean loss $\ell(\theta',\theta) = \Vert \theta'-\theta\Vert_2$ and for $\epsilon_n = (\sigma/8) \sqrt{d/n}$,
\[
\inf_{\theta \in \R^d} \E_{\theta}\Bigl[\P_{\pi'}\bigl(\theta': \Vert \theta'-\theta\Vert_2 \leq \epsilon_n \, \big| \, X_{1:n} \bigr)\Bigr] \leq c_d \,,
\]
where $(c_d)_{d \geq 1}$ is a decreasing sequence such that $c_1 \leq 0.55$, $c_2 \leq 0.37$, and
$c_d \to 0.21$ as $d \to +\infty$.
\end{proposition}

This proposition indicates that the best possible posterior concentration rate is at best $\sigma \sqrt{d/n}$ up to a multiplicative constant; actually, this order of magnitude is the best achievable posterior concentration rate,
see, e.g., \citet[Chapter~8]{LeCamYang-00-AsymptoticsinStatistics}. \medskip

There are at least two ways to prove the lower bound of Proposition~\ref{prop:lowerbound-posterior-minimax}. A first one is to use a well-known conversion of ``good'' Bayesian posteriors into ``good'' point estimators, which indicates that lower bounds for point estimation can be turned into lower bounds for posterior concentration. For the sake of completeness, we recall this conversion in Appendix~\ref{sec:bayesianToPoint} and provide a nonasymptotic variant of Theorem~2.5 by \citet{GhosalETAL-00-ConvergenceRatesPosterior}.

The second method---followed in the proof below---is however more direct. We use our most general continuous Fano's inequality with the random variables $Z_{\theta} =  \P_{\pi'}\bigl(\theta': \Vert \theta'-\theta\Vert_2 \leq \epsilon_n \, \big| \, X_{1:n} \bigr) \in [0,1]$. \bigskip

\begin{proof}
We may assume, with no loss of generality, that the probability space on which $X_{1:n}$ is defined is $(\R^d)^n$ endowed with its Borel $\sigma$--field and the probability measure $\P_\theta = \cN(\theta,\sigma^2)^{\otimes n}$.
Let $\nu$ denote the uniform distribution on the Euclidean ball $B(0,\rho \epsilon_n) = \bigl\{ u \in \R^d : \Vert u \Vert_2 \leq \rho \epsilon_n \bigr\}$ for some $\rho>1$ to be determined by the analysis. Then, by the continuous Fano inequality in the form given by the combination of \eqref{eq:red5} and \eqref{eq:bd3}, with $\Q_{\theta} = \P_0 = \cN(0,\sigma^2)^{\otimes n}$, where $0$ denotes the null vector of $\R^d$, and with the $[0,1]$--valued random variables $Z_{\theta} =  \P_{\pi'}\bigl(\theta': \Vert \theta'-\theta\Vert_2 \leq \epsilon_n \, \big| \, X_{1:n} \bigr)$, we have
\begin{align}
\inf_{\theta \in \R^d} \E_{\theta}\bigl[Z_{\theta}\bigr] \leq \int_{B(0,\rho \epsilon_n)} \E_{\theta}\bigl[Z_{\theta}\bigr] \d\nu(\theta) & \leq \int_{B(0,\rho \epsilon_n)} \E_{0} \bigl[ Z_{\theta} \bigr] \d\nu(\theta) + \sqrt{\frac{\displaystyle{\int_{B(0,\rho \epsilon_n)} \KL(\P_{\theta},\P_0)\d\nu(\theta)}}{- \displaystyle{\ln \int_{B(0,\rho \epsilon_n)} \E_{0} \bigl[ Z_{\theta} \bigr] \d\nu(\theta)}}} \nonumber \\
& \leq \left(\frac{1}{\rho}\right)^{\! d}
+ \sqrt{\frac{n \rho^2 \epsilon_n^2 / (2 \sigma^2)}{d \ln \rho}}\,,
\label{eq:lowerbound-posterior-Gaussian-1bis}
\end{align}
where the last inequality follows from \eqref{eq:KLupperbound} and \eqref{eq:lowerbound-posterior-Gaussian-3} below. First note that, by independence, $\KL \bigl(\P_{\theta}, \P_0 \bigr) = n \KL\bigl( \cN(\theta,\sigma^2), \, \cN(0,\sigma^2) \bigr) = n \norm{\theta}_2^2 / (2 \sigma^2)$,
so that
\begin{equation}
\bigintsss_{B(0,\rho \epsilon_n)} \KL \bigl(\P_{\theta}, \P_0 \bigr) \d\nu(\theta) =
\frac{n}{2 \sigma^2} \bigintsss_{B(0,\rho \epsilon_n)} \norm{\theta}_2^2 \d\nu(\theta)
\leq
\frac{n \rho^2 \epsilon_n^2}{2 \sigma^2} \,.
\label{eq:KLupperbound}
\end{equation}

\noindent
Second, using the Fubini-Tonelli theorem (twice) and the definition of \[ Z_{\theta} = \P_{\pi'}\bigl(\theta': \Vert \theta'-\theta\Vert_2 \leq \epsilon_n \, \big| \, X_{1:n} \bigr) = \E_{\theta' \sim \P_{\pi'}(\cdot \,\, | \, X_{1:n})}\bigl[\indicator{\{\Vert \theta'-\theta\Vert_2 \leq \epsilon_n\}}\bigr]\,, \] we can see that
\begin{align}
q \defeq \bigintsss_{B(0,\rho \epsilon_n)} \E_{0} \bigl[ Z_{\theta} \bigr] \d\nu(\theta) & = \E_0\!\left[ \bigintsss_{B(0,\rho \epsilon_n)} \E_{\theta' \sim \P_{\pi'}(\cdot \,\, | \, X_{1:n})}\bigl[\indicator{\{\Vert \theta'-\theta\Vert_2 \leq \epsilon_n\}}\bigr] \d\nu(\theta) \right] \nonumber \\
& = \E_0\!\left[ \E_{\theta' \sim \P_{\pi'}(\cdot \,\, | \, X_{1:n})} \!\left[ \bigintsss_{B(0,\rho \epsilon_n)} \indicator{\{\Vert \theta'-\theta\Vert_2 \leq \epsilon_n\}} \d\nu(\theta) \right] \right] \nonumber \\
& = \E_0\biggl[ \E_{\theta' \sim \P_{\pi'}(\cdot \,\, | \, X_{1:n})} \Bigl[ \nu\bigl(B(\theta',\epsilon_n) \cap B(0,\rho \epsilon_n) \bigr) \Bigr] \biggr] \leq \left(\frac{1}{\rho}\right)^d \,, \label{eq:lowerbound-posterior-Gaussian-3}
\end{align}
where to get the last inequality we used the fact that $\nu\bigl(B(\theta',\epsilon_n) \cap B(0,\rho \epsilon_n) \bigr)$ is the ratio of the volume of the (possibly truncated) Euclidean ball $B(\theta',\epsilon_n)$ of radius $\epsilon_n$ and center $\theta'$ with the volume of the support of~$\nu$, namely, the larger Euclidean ball $B(0,\rho \epsilon_n)$, in dimension $d$.

The proof is then concluded by recalling that $\rho > 1$ was a parameter of the analysis and by picking, e.g.,
$\epsilon_n = (\sigma/8) \sqrt{d/n}$: by~\eqref{eq:lowerbound-posterior-Gaussian-1bis}, we have
\begin{align*}
\inf_{\theta \in \R^d} \E_{\theta}\Bigl[\P_{\pi}\bigl(\theta': \Vert \theta'-\theta\Vert_2 \leq \epsilon_n \, \big| \, X_{1:n} \bigr)\Bigr] =
\inf_{\theta \in \R^d} \E_{\theta}\bigl[Z_{\theta}\bigr]
\leq \inf_{\rho > 1} \left\{\left(\frac{1}{\rho}\right)^{\! d}
+ \frac{\rho}{8 \sqrt{2 \ln \rho}} \, \right\} \defeq c_d \,.
\end{align*}
We can see that $c_1 \leq 0.55$ and $c_2 \leq 0.37$ via the respective choices
$\rho = 5$ and $\rho = 3$,
while the fact that the limit is smaller than (and actually equal to) $\sqrt{e}/8 \leq 0.21$ follows from the choice $\rho = \sqrt{e}$.

Note that, when using~\eqref{eq:bd3} above, we implicitly assumed that the quantity $q$ in~\eqref{eq:lowerbound-posterior-Gaussian-3} lies in $(0,1)$. The fact that $q<1$ follows directly from the upper bound $(1/\rho)^d$ and from $\rho>1$. Besides, the condition $q>0$ is met as soon as $\P_0\bigl(\P_{\pi'}(B(0,\epsilon_n) \,\, | \, X_{1:n}) > 0\bigr) > 0$; indeed, for $\theta' \in
B(0,\epsilon_n)$, we have $\nu\bigl(B(\theta',\epsilon_n) \cap B(0,\rho \epsilon_n) \bigr) > 0$
and thus $q$ appears in the last equality of~\eqref{eq:lowerbound-posterior-Gaussian-3} as being lower bounded by the expectation of a positive function over a set with positive probability.
If on the contrary $\P_0\bigl(\P_{\pi'}(B(0,\epsilon_n) \,\, | \, X_{1:n}) > 0\bigr) = 0$, then $\P_0(Z_0>0)=0$, so that $\inf_{\theta} \E_{\theta}[Z_{\theta}] = \E_0[Z_0] = 0$, which immediately implies the bound of Proposition~\ref{prop:lowerbound-posterior-minimax}.
\end{proof}

\begin{remark} {\em
Though the lower bound of Proposition~\ref{prop:lowerbound-posterior-minimax} is only stated for the posterior distributions $\P_{\pi'}(\,\cdot\,\,|\,X_{1:n})$, it is actually valid for any transition kernel $Q(\,\cdot\,\,|\,X_{1:n})$. This is because the proof above relies on general information-theoretic arguments and does not use the particular form of $\P_{\pi'}(\,\cdot\,\,|\,X_{1:n})$. This is in the same spirit as for minimax lower bounds for point estimation.
}
\end{remark}

In Section~\ref{sec:bayesianPosteriorRates-distribDependent} we derive another type of posterior concentration lower bound that is no longer uniform. More precisely, we prove a distribution-dependent lower bound that specifies how the posterior mass fails to concentrate on $\epsilon_n$--neighborhoods of $\theta$ for every $\theta \in \Theta$.

\subsection{Lower bounds in robust sequential learning with sparse losses}
\label{sec:usecase-sequential}

We consider a framework of robust sequential learning called prediction of individual
sequences. Its origins and core results are described in the monography by \citet{CBL06}.
In its simplest version, a decision-maker and an environment play repeatedly as follows:
at each round $t \geq 1$, and simultaneously, the environment chooses a vector of losses
$\ell_t = (\ell_{1,t},\ldots,\ell_{N,t}) \in [0,1]^N$ while the decision-maker picks
an index $I_t \in \{1,\ldots,N\}$, possibly at random. Both players
then observe $\ell_t$ and~$I_t$. The decision-maker wants to minimize her cumulative
regret, the difference between her cumulative loss and the cumulative loss associated with
the best constant choice of an index: for $T \geq 1$,
\[
R_T = \sum_{t=1}^T \ell_{I_t,t} - \min_{k=1,\ldots,N} \sum_{t=1}^T \ell_{k,t}\,.
\]
In this setting the optimal regret in the worst-case is of the order of $\sqrt{T \ln(N)}$. \citet{CBFH+97} exhibited an asymptotic lower bound of $\sqrt{T \ln(N)/2}$, based on the central limit theorem and on the fact that the expectation
of the maximum of $N$ independent standard Gaussian random variables is of the order of $\sqrt{\ln(N)}$.
To do so, they considered stochastic environments drawing independently the loss
vectors $\ell_t$ according to a well-chosen distribution.

\citet{CBLS05} extended this result to a variant called label-efficient prediction, in
which loss vectors are observed upon choosing and with a budget constraint: no more than
$m$ observations within $T$ rounds. They prove an optimal and non-asymptotic lower bound
on the regret of the order of $T \sqrt{\ln(N)/m}$, based on several applications of
Fano's inequality to deterministic strategies of the decision-maker, and then, an application of Fubini's
theorem to handle general, randomized, strategies.
Our re-shuffled proof technique below shows that a single application
of Fano's inequality to general strategies would be sufficient there (details omitted).

Recently, \citet{KwPe15} considered a setting of sparse loss vectors, in which
at each round at most $s$ of the $N$ components of the loss vectors $\ell_t$ are
different from zero. They prove an optimal and asymptotic lower bound
on the regret of the order of $\sqrt{T s \ln(N)/N}$, which generalizes
the result for the basic framework, in which $s = N$. Their proof
is an extension of the proof of \citet{CBFH+97} and is based on
the central limit theorem together with additional technicalities,
e.g., the use of Slepian's lemma to deal with some dependencies arising from
the sparsity assumption.

The aim of this section is to provide a short and elementary proof of
this optimal asymptotic $\sqrt{T s \ln(N)/N}$ bound. As a side result, our bound will
even be non-asymptotic. However, for small values of $T$, given that $s/N$ is small,
picking components $I_t$ uniformly at random ensures an expected cumulative loss thus an expected cumulative regret
less than $sT/N$. The latter is smaller than $\sqrt{T s \ln(N)/N}$ for values
of $T$ of the order of $N \ln(N)/s$.
This is why the bound below involves a minimum between quantities of the order
of $\sqrt{T s \ln(N)/N}$ and $sT/N$; it matches the upper bounds on the regret that can
be guaranteed and is therefore optimal.

The expectation in the statement below
is with respect to the internal randomization used by the
decision-maker's strategy.

\begin{theorem}
For all strategies of the decision-maker, for all $s \in \{0,\ldots,N\}$,
for all $N \geq 2$, for all $T \geq 1$,
there exists a fixed-in-advance
sequence of loss vectors $\ell_1,\ldots,\ell_T$ in $[0,1]^N$ that are each $s$--sparse such that
\[
\E\bigl[R_T\bigr] = \sum_{t=1}^T \E\bigl[\ell_{I_t,t}\bigr]
- \min_{k=1,\ldots,N} \sum_{t=1}^T \ell_{k,t} \geq
\min\left\{ \frac{s}{16 N} T, \,\, \frac{1}{32} \sqrt{T \frac{s}{N} \ln N} \right\}.
\]
\end{theorem}

\begin{proof}
The case $s=0$ corresponds to instantaneous losses $\ell_{j,t}$ that are all null,
so that the regret is null as well. Our lower bound holds in this case, but is uninteresting.
We therefore focus in the rest of this proof on the case $s \in \{1,\ldots,N\}$.

We fix $\varepsilon \in \bigl(0, \, s/(2N) \bigr)$
and consider, as \citet{KwPe15} did, independent and identically distributed loss vectors $\ell_t \in [0,1]^N$,
drawn according to one distribution among $P_i$, where $1 \leq i \leq N$. Each distribution $P_i$ on $[0,1]^N$ is defined as the law of a random vector $L$ drawn in two steps as follows. We pick $s$ components uniformly at random among $\{1,\ldots,N\}$.
Then, the components $k$ not picked are associated with zero losses, $L_k = 0$.
The losses $L_k$ for picked components $k \ne i$ are drawn according to a Bernoulli distribution
with parameter $1/2$. If component $i$ is picked, its loss $L_i$ is drawn according to
a Bernoulli distribution with parameter $1/2 - \varepsilon N / s$.
The loss vector $L \in [0,1]^N$ thus generated is indeed $s$--sparse.
We denote by $P_i^T$ the $T$--th product distribution $P_i \otimes \cdots \otimes P_i$.
We will actually identify the underlying probability and the law $P_i^T$.
Finally, we denote the expectation under $P_i^T$ by $\E_i$.

Now, under $P_i^T$, the components $\ell_{k,t}$ of the loss vectors
are all distributed according to Bernoulli distributions,
with parameters $s/(2N)$ if $k \ne i$ and $s/(2N) - \varepsilon$ if $k = i$.
The expected regret, where the expectation $\E$ is with respect to the
strategy's internal randomization and the expectation $\E_i$ is with respect
to the random choice of the loss vectors, is thus larger than
\begin{align}
\nonumber
\E_i\Bigl[ \E\bigl[R_T\bigr] \Bigr]
\geq \sum_{t=1}^T \E_i \Bigl[ \E \bigl[ \ell_{I_t,t}\bigr] \Bigr]
- \min_{k=1,\ldots,N} \sum_{t=1}^T \E_i \bigl[ \ell_{k,t}\bigr]
& = \sum_{t=1}^T \frac{s}{2N}
\biggl( 1 - \varepsilon \, \E_i \Bigl[ \E \bigl[ \ind{ {\{ I_t = i \}} } \bigr] \Bigr] \biggr)
- T \biggl( \frac{s}{2N} - \varepsilon \biggr) \\
\nonumber
& = T \varepsilon \biggl( 1 - \E_i \Bigl[ \E \bigl[ F_i(T) \bigr] \Bigr] \biggr)\,,
\end{align}
where
\[
F_i(T) = \frac{1}{T} \sum_{t=1}^T \ind{\{ I_t = i \}}\,.
\]
All in all, we copied almost word for word the (standard) beginning of the proof by \citet{KwPe15},
whose first lower bound is exactly
\begin{equation}
\label{eq:suitesindsparse-1}
\sup_{\ell_1,\ldots,\ell_t} \E\bigl[R_T\bigr]
\geq \frac{1}{N} \sum_{i=1}^N \E_i\Bigl[ \E\bigl[R_T\bigr] \Bigr]
\geq T \varepsilon \biggl( 1 - \frac{1}{N} \sum_{i=1}^N \E_i \Bigl[ \E \bigl[ F_i(T) \bigr] \Bigr] \biggr).
\end{equation}
The main differences arise now: we replace a long asymptotic argument (based on the central limit theorem
and the study of the limit via Slepian's lemma) by a single application of Fano's inequality.

We introduce the distribution $Q$ on $[0,1]^N$ corresponding to the same randomization
scheme as for the $P_i$, except that no picked component is favored
and that all their corresponding losses are drawn according to
the Bernoulli distribution with parameter $1/2$.  We also denote by $\P$ the probability distribution that underlies the internal randomization of the strategy.
An application of Lemma~\ref{lm:Fano-suitesind} with $\P_i = \P \otimes P_i^T$ and $\Q_i = \P \otimes Q^T$, using that $F_1(T) + \ldots + F_N(T) = 1$ and thus $(1/N) \sum_{i=1}^N \E_{Q} \bigl[ \E[ F_i(T) ] \bigr] = 1/N$,
yields
\begin{equation}
\label{eq:suitesindsparse-2}
\frac{1}{N} \sum_{i=1}^N \E_i \Bigl[ \E \bigl[ F_i(T) \bigr] \Bigr]
\leq \frac{1}{N} + \sqrt{\frac{1}{N \ln(N)} \sum_{i=1}^N \KL\bigl(\P \otimes P_i^T, \, \P \otimes Q^{T}\bigr)}\,.
\end{equation}
By independence, we get, for all $i$,
\begin{equation}
\label{eq:suitesindsparse-2bis}
\KL\bigl(\P \otimes P_i^T, \, \P \otimes Q^T\bigr)
= \KL\bigl(P_i^T, \, Q^T\bigr)
= T \, \KL\bigl(P_i, \, Q\bigr)\,.
\end{equation}
We now show that
\begin{equation}
\label{eq:suitesindsparse-3}
\KL\bigl(P_i, \, Q\bigr)
\leq \frac{s}{N} \, \kl\biggl( \frac{1}{2}- \varepsilon \frac{N}{s},\,\frac{1}{2}\biggr)\,.
\end{equation}
Indeed, both $P_i$ and $Q$ can be seen
as uniform convex combinations
of probability distributions of the following form,
indexed by the subsets of $\{1,\ldots,N\}$ with $s$ elements and
up to permutations
of the Bernoulli distributions in the products below (which does not
change the value of the Kullback-Leibler divergences between them):
\begin{multline*}
\binom{N-1}{s-1} \ \ \mbox{distributions of the form (when $i$ is picked)} \\
\Ber\biggl( \frac{1}{2} - \varepsilon \frac{N}{s} \biggr) \otimes \bigotimes_{k=2}^s
\Ber\biggl( \frac{1}{2} \biggr) \otimes \bigotimes_{k=s+1}^N \delta_0
\qquad \mbox{and} \qquad
\bigotimes_{k=1}^s
\Ber\biggl( \frac{1}{2} \biggr) \otimes \bigotimes_{k=s+1}^N \delta_0\,,
\end{multline*}
where $\delta_0$ denotes the Dirac mass at $0$,
and
\begin{multline*}
\binom{N-1}{s} \ \ \mbox{distributions of the form (when $i$ is not picked)} \\
\bigotimes_{k=1}^s
\Ber\biggl( \frac{1}{2} \biggr) \otimes \bigotimes_{k=s+1}^N \delta_0
\qquad \mbox{and} \qquad
\bigotimes_{k=1}^s
\Ber\biggl( \frac{1}{2} \biggr) \otimes \bigotimes_{k=s+1}^N \delta_0\,.
\end{multline*}
Only the first set of distributions contributes to the Kullback-Leibler divergence.
By convexity of the
Kullback-Leibler divergence (Corollary~\ref{cor:jointconvKL--fdiv}), we
thus get the inequality
\begin{align*}
\KL\bigl(P_i, \, Q\bigr)
& \leq \frac{\displaystyle{\binom{N-1}{s-1}}}{\displaystyle{\binom{N}{s}}}
\KL \Biggl( \Ber\biggl( \frac{1}{2} - \varepsilon \frac{N}{s} \biggr) \otimes \bigotimes_{k=2}^s
\Ber\biggl( \frac{1}{2} \biggr) \otimes \bigotimes_{k=s+1}^N \delta_0, \,\,
\bigotimes_{k=1}^s
\Ber\biggl( \frac{1}{2} \biggr) \otimes \bigotimes_{k=s+1}^N \delta_0 \Biggr) \\
& = \frac{s}{N} \, \kl\biggl( \frac{1}{2}- \varepsilon \frac{N}{s},\,\frac{1}{2}\biggr)\,,
\end{align*}
where the last equality is again by independence.
Finally, the lemma stated right after this proof shows that
\begin{equation}
\label{eq:suitesindsparse-4}
\kl\biggl( \frac{1}{2}- \varepsilon \frac{N}{s} ,\,\frac{1}{2}\biggr)
\leq \frac{4 N^2 \varepsilon^2}{s^2} \,.
\end{equation}

Combining \eqref{eq:suitesindsparse-1}--\eqref{eq:suitesindsparse-4}, we proved so far
\[
\forall \varepsilon \in \bigl(0, \, s/(2N) \bigr), \qquad
\sup_{\ell_1,\ldots,\ell_t} \E\bigl[R_T\bigr]
\geq T \varepsilon \biggl( 1 - \frac{1}{N} - \sqrt{\frac{4 N T \epsilon^2}{s \ln(N)}} \biggr)
\geq T \varepsilon \biggl( \frac{1}{2} - c\,\varepsilon \biggr)\,,
\]
where we used $1/N \leq 1/2$ and denoted $c = 2 \sqrt{N T} \big/ \sqrt{s \ln(N)}$.

A standard optimization suggests the choice $\epsilon = 1/(4c)$, which is valid, i.e., is
indeed $< s/(2N)$ as required, as soon as $T > N \ln(N) / (16 s)$. In that case, we
get a lower bound $T\epsilon/4$, which corresponds to the
$\sqrt{T s \ln(N)/N}/32$ part of the lower bound.

In case $T \leq N \ln(N) / (16 s)$, we have $c \leq N/(2s)$ and the valid choice $\varepsilon = s/(4N)$
leads to the part of the lower bound given by $T\varepsilon(1/2 - c\varepsilon)
\geq T\varepsilon/4 = sT/(16N)$.
\end{proof}

\begin{lemma}
For all $p \in (0,1)$, for all $\epsilon \in (0,p)$,
\[
\kl(p-\epsilon,p) \leq \frac{\epsilon^2}{p(1-p)}\,.
\]
\end{lemma}

\begin{proof}
This result is a special case of the fact that the $\KL$ divergence is upper bounded by the $\chi^2$--divergence. We recall, in our particular case, how this is seen:
\[
\kl(p-\epsilon,p) =
(p-\epsilon) \ln \biggl( 1 - \frac{\epsilon}{p} \biggr)
+ (1-p+\epsilon) \ln \biggl( 1 + \frac{\epsilon}{1-p} \biggr)
\leq (p-\epsilon) \frac{- \epsilon}{p} + (1-p+\epsilon) \frac{\epsilon}{1-p}
= \frac{\epsilon^2}{p} + \frac{\epsilon^2}{1-p}\,,
\]
where we used $\ln(1+u) \leq u$ for all $u > -1$ to get the stated inequality.
\end{proof}

\newpage
\section{Other applications, with $N=1$ pair of distributions}
\label{sec:otherapplications}

Interestingly, Proposition~\ref{prop:FanoIT} can be useful even for $N=1$ pair of distributions. Rewriting it slightly differently, we indeed have, for all distributions $\P,\Q$ and all events $A$ with $\Q(A) \in (0,1)$,
\[
\P(A) \ln\!\left(\frac{1}{\Q(A)}\right) \leq \KL(\P,\Q) + \ln(2) \,.
\]
Solving for $\Q(A)$---and not for $\P(A)$ as was previously the case---we get
\begin{equation}
\Q(A) \geq \exp\!\left(- \frac{\KL(\P,\Q) + \ln(2)}{\P(A)}\right) \,.
\label{eq:Fano1}
\end{equation}
We applied here a classical technique in information theory due to Haroutunian;
see, for instance, \citet[page~167]{CK81}.
The inequality above also holds in the case $\Q(A) = 1$, as the right-hand side is the
exponential of a nonpositive quantity, and in the case $\Q(A) = 0$.
Indeed, we either have $\P(A) > 0$, which entails, by
the data-processing inequality (Lemma~\ref{lem:dataProcessingIneq--fdiv}),
\[
\KL(\P,\Q) \geq \kl\bigl(\P(A),\Q(A)\bigr) = +\infty\,,
\]
or $\P(A)=0$; that is, when $\Q(A) = 0$, no matter the value of $\P(A)$, the inequality
features the exponential of $-\infty$ in its right-hand side.

Similarly and more generally, for all distributions $\P,\Q$ and all $[0,1]$--valued random variables $Z$, we have, by Corollary~\ref{lem:dataProcessingIneq2--fdiv}
and the lower bound~\eqref{eq:bd1-intro},
\begin{equation}
\E_{\Q}[Z] \geq \exp\!\left(- \frac{\KL(\P,\Q) + \ln(2)}{\E_{\P}[Z]}\right) \,.
\label{eq:Fano1bis}
\end{equation}

The bound~\eqref{eq:Fano1} is similar in spirit to (a consequence of) the Bretagnolle-Huber inequality, recalled and actually improved in Section~\ref{sec:proofLBKL-BH}; see details therein, and in particular its consequence~\eqref{eq:bd4bis}. Both bounds can indeed be useful when $\KL(\P,\Q)$ is larger than a constant and $\P(A)$ is close to~$1$. \smallskip

Next we show two applications of \eqref{eq:Fano1} and~\eqref{eq:Fano1bis}: a simple proof of a large deviation lower bound for Bernoulli distributions, and a distribution-dependent posterior concentration lower bound.

\subsection{A simple proof of Cram\'{e}r's theorem for Bernoulli distributions}
\label{sec:Cramerlowerbound}

The next proposition is a well-known large deviation result on the sample mean of independent and identically distributed Bernoulli random variables. It is a particular case of Cram\'{e}r's theorem that dates back to \citet{Cra-38-TheoremeCramer,Che-52-TheoremeCramer}; see also \citet{CePe-11-ShortProofCramerR} for further references and a proof in a very general context. Thanks to Fano's inequality~\eqref{eq:Fano1}, the proof of the lower bound that we provide below
avoids any explicit change of measure (see the remark after the proof). We are grateful to Aur\'{e}lien Garivier for suggesting this proof technique to us; see also strong connections with an approach followed by
\citet[Section~2.4.2]{Hayashi17}.

\begin{proposition}[Cram\'{e}r's theorem for Bernoulli distributions]
\label{prop:lowerbound-largedeviations}
Let $\theta \in (0,1)$. Assume that $X_1,\ldots,X_n$ are independent and identically distributed random variables drawn from $\Ber(\theta)$. Denoting by $\P_{\theta}$ the underlying probability measure, we have, for all $x \in (\theta,1)$,
\[
\lim_{n \to + \infty} \, \frac{1}{n} \ln \P_{\theta}\!\left(\frac{1}{n} \sum_{i=1}^n X_i > x \right) = -\kl(x,\theta) \,.
\]
\end{proposition}

\begin{proof}
We set $\ol{X}_n \defeq n^{-1} \sum_{i=1}^n X_i$. For the convenience of the reader we first briefly recall how to prove the upper bound, and then proceed with a new proof for the lower bound.

\ \\
\textit{Upper bound}: By the Cram\'{e}r-Chernoff method and the duality formula for the Kullback-Leibler divergence between Bernoulli distributions (see, e.g., \citealt[pages 21--24]{BoLuMa12}), we have, for all $n \geq 1$,
\begin{equation}
\label{eq:CramerChernoff-upper}
\P_{\theta}\bigl(\ol{X}_n > x\bigr) \leq \exp\Biggl(-n \sup_{\lambda>0} \biggl\{\lambda x - \ln \E_{\theta}\Bigr[e^{\lambda X_1}\Bigr]\biggr\}\Biggr) = \exp\Bigl(-n\kl(x,\theta)\Bigr) \,,
\end{equation}
that is,
\[
\forall n \geq 1, \qquad \frac{1}{n} \ln \P_{\theta}\bigl(\ol{X}_n > x\bigr) \leq - \kl(x,\theta)\,.
\]

\ \\
\textit{Lower bound}: Choose $\epsilon>0$ small enough such that $x+\epsilon <1$. We may assume with no loss of generality that the underlying distribution is $\P_{\theta} = \Ber(\theta)^{\otimes n}$.
By Fano's inequality in the form~\eqref{eq:Fano1} with the distributions $\P=\P_{x+\epsilon}$ and $\Q = \P_{\theta}$, and the event $A = \bigl\{\ol{X}_n > x \bigr\}$, we have
\[
\P_{\theta}\bigl(\ol{X}_n > x\bigr) \geq \exp\!\left(- \frac{\KL(\P_{x+\epsilon},\P_{\theta}) + \ln(2)}{\P_{x+\epsilon}\bigl(\ol{X}_n > x\bigr)}\right) \,.
\]

\noindent
Noting that $\KL(\P_{x+\epsilon},\P_{\theta}) = n \kl(x+\epsilon,\theta)$ we get
\begin{align}
\P_{\theta}\bigl(\ol{X}_n > x\bigr) & \geq \exp\!\left(- \frac{n \kl(x+\epsilon,\theta) + \ln 2}{\P_{x+\epsilon}\bigl(\ol{X}_n > x\bigr)} \right) \geq \exp\!\left(- \frac{n \kl(x+\epsilon,\theta) + \ln 2}{1-e^{-n \kl(x,x+\epsilon)}} \right) ,
\end{align}
where the last bound follows from $\P_{x+\epsilon}\bigl(\ol{X}_n > x\bigr) = 1 - \P_{x+\epsilon}\bigl(\ol{X}_n \leq x\bigr) \geq 1 - e^{-n \kl(x,x+\epsilon)}$ by a derivation similar to~\eqref{eq:CramerChernoff-upper} above. Taking the logarithms of both sides and letting $n \to +\infty$ finally yields
\[
\liminf_{n \to + \infty} \, \frac{1}{n} \ln \P_{\theta}\bigl(\ol{X}_n > x\bigr) \geq - \kl(x+\epsilon,\theta)\,.
\]
We conclude the proof by letting $\epsilon \to 0$, and by combining the upper and lower bounds.
\end{proof}

\paragraph{Comparison with an historical proof.} A classical proof for the lower bound relies on the same change of measure as the one used above, i.e., that transports the measure $\Ber(\theta)^{\otimes n}$ to $\Ber(x+\epsilon)^{\otimes n}$. The bound~\eqref{eq:CramerChernoff-upper}, or any other large deviation inequality, is also typically used therein. However, the change of measure is usually carried out explicitly by writing
\[
\P_{\theta}\bigl(\ol{X}_n > x\bigr) = \E_{\theta}\!\left[\indicator{\bigl\{\ol{X}_n > x\bigr\}}\right] = \E_{x+\epsilon}\!\left[\indicator{\bigl\{\ol{X}_n > x\bigr\}} \frac{\d \P_{\theta}}{\d \P_{x+\epsilon}}(X_1,\ldots,X_n) \right] = \E_{x+\epsilon}\!\left[\indicator{\bigl\{\ol{X}_n > x\bigr\}} e^{-n \,\hat{\KL}_n}\right] \,,
\]
where the empirical Kullback-Leibler divergence $\hat{\KL}_n$ is defined by
\[
\hat{\KL}_n \defeq \frac{1}{n} \ln\!\left(\frac{\d \P_{x+\epsilon}}{\d \P_{\theta}}(X_1,\ldots,X_n)\right) = \frac{1}{n} \sum_{i=1}^n \Biggl( \indicator{\{X_i = 1\}} \ln\!\left(\frac{x+\epsilon}{\theta}\right) + \indicator{\{X_i = 0\}} \ln\!\left(\frac{1-(x+\epsilon)}{1-\theta} \right) \Biggr) \,.
\]
The empirical Kullback-Leibler divergence $\hat{\KL}_n$ is then compared to its limit $\kl(x+\epsilon,\theta)$ via the law of large numbers. On the contrary, our short proof above bypasses any call to the law of large numbers and does not perform the change of measure explicitely,
in the same spirit as for the bandit lower bounds derived by \citet{KaCaGa-16-BestArmIdentification} and \citet{GaMeSt16}.
Note that the different and more general proof of \citet{CePe-11-ShortProofCramerR} also bypassed any call to the law of large numbers thanks to other convex duality arguments.

\subsection{Distribution-dependent posterior concentration lower bounds}
\label{sec:bayesianPosteriorRates-distribDependent}

In this section we consider the same Bayesian setting as the one described at the beginning of Section~\ref{sec:bayesianPosteriorRates}. In addition, we define the global modulus of continuity between $\KL$ and $\ell$ around $\theta \in \Theta$ and at scale $\epsilon_n > 0$ by
\[
\psi\bigl(\epsilon_n,\theta,\ell\bigr) \defeq \inf \Bigl\{\KL\bigl(P_{\theta'},P_{\theta}\bigr): \, \ell(\theta',\theta) \geq 2 \epsilon_n, \ \theta' \in \Theta \Bigr\} \,;
\]
the infimum is set to $+\infty$ if the set is empty.

Next we provide a distribution-dependent lower bound for posterior concentration rates, that is, a lower bound that holds true for every $\theta \in \Theta$, as opposed\footnote{Note however that we are here in a slightly different regime than in Section~\ref{sec:bayesianPosteriorRates}, where we addressed cases for which the uniform posterior concentration condition~\eqref{eq:posteriorlowerbound-DD-condition} was proved to be impossible at scale $\epsilon_n$ (and actually took place at a slightly larger scale $\epsilon_n'$).} to the minimax lower bound of Section~\ref{sec:bayesianPosteriorRates}. Theorem~\ref{thm:lowerbound-posterior-distribDependent} below indicates that, if the $\ell$--ball around $\theta$ with radius $\epsilon_n$ has an expected posterior mass close to $1$ uniformly over all $\theta \in \Theta$, then this posterior mass cannot be too close to $1$ either. Indeed, Inequality~\eqref{eq:posteriorlowerbound-DD-conclusion} provides a lower bound on the expected posterior mass outside of this ball. The term $n \, \psi(\epsilon_n,\theta,\ell)$ within the exponential is a way to quantify how difficult it can be to distinguish between the two product measures $P_{\theta'}^{\otimes n}$ and $P_{\theta}^{\otimes n}$ when $\ell(\theta',\theta) \geq 2 \epsilon_n$.

\begin{theorem}[Distribution-dependent posterior concentration lower bound]
\label{thm:lowerbound-posterior-distribDependent}
Assume that the posterior distribution $\P_{\pi}(\,\cdot\,\,|\,X_{1:n})$ satisfies the uniform concentration condition
\begin{equation*}
\inf_{\theta \in \Theta} \,\, \E_{\theta}\Bigl[\P_{\pi}\bigl(\theta': \ell(\theta',\theta) < \epsilon_n \, \big| \, X_{1:n} \bigr)\Bigr] \longrightarrow 1 \qquad \textrm{as } n \to +\infty\,.
\end{equation*}
Then,
for all $c > 1$, for all $n$ large enough, for all $\theta \in \Theta$,
\begin{equation}
\E_{\theta}\Bigl[\P_{\pi}\bigl(\theta': \ell(\theta',\theta) > \epsilon_n \,\big|\, X_{1:n} \bigr)\Bigr] \geq 2^{-c} \, \exp\Bigl(- c \, n \, \psi(\epsilon_n,\theta,\ell)\Bigr) \,.
\label{eq:posteriorlowerbound-DD-conclusion}
\end{equation}
\end{theorem}

\noindent
The conclusion can be stated equivalently as:
\[
\liminf_{n \to +\infty} \,\,\,
\inf_{\theta \in \Theta} \,
\frac{\ln \biggl( \E_{\theta}\Bigl[\P_{\pi}\bigl(\theta': \ell(\theta',\theta) > \epsilon_n \big| X_{1:n} \bigr)\Bigr] \biggr)}{\ln(2) + n \, \psi(\epsilon_n,\theta,\ell)} \geq - 1\,.
\]
The above theorem is greatly inspired from Theorem~2.1 by \citet{HoRoSH-15-AdaptivePosteriorConcentrationRates}. Our Fano's inequality~\eqref{eq:Fano1bis} however makes the proof more direct: the change-of-measure carried out by \citet{HoRoSH-15-AdaptivePosteriorConcentrationRates} is now implicit, and no proof by contradiction is required. We also bypass one technical assumption (see the discussion after the proof). \\

\begin{proof}
We fix $c>1$. By the uniform concentration condition, there exists $n_0 \geq 1$ such that, for all $n \geq n_0$,
\begin{equation}
\inf_{\theta^\star \in \Theta} \,\, \E_{\theta^\star}\Bigl[\P_{\pi}\bigl(\theta': \ell(\theta',\theta^\star) < \epsilon_n \,\big|\, X_{1:n} \bigr)\Bigr] \geq \frac{1}{c} \;.
\label{eq:posteriorlowerbound-DD-condition}
\end{equation}
We now fix $n \geq n_0$ and $\theta \in \Theta$. We consider any $\theta^\star \in \Theta$ such that $\ell\bigl(\theta^\star,\theta\bigr) \geq 2 \epsilon_n$. Using Fano's inequality in the form of~\eqref{eq:Fano1bis} with the distributions
$\P=P_{\theta^\star}^{\otimes n}$ and $\Q = P_{\theta}^{\otimes n}$, together with the $[0,1]$--valued random variable $Z_{\theta} = \P_{\pi}\bigl(\theta': \ell(\theta',\theta) > \epsilon_n \,\big|\, X_{1:n} \bigr)$, we get
\begin{equation}
\E_{\theta}\bigl[Z_{\theta}\bigr] \geq  \exp\!\left( - \frac{\KL\Bigl(P_{\theta^\star}^{\otimes n},P_{\theta}^{\otimes n}\Bigr) + \ln 2}{\E_{\theta^\star}\bigl[Z_{\theta}\bigr]} \right) = \exp\!\left( - \frac{n \KL\bigl(P_{\theta^\star},P_{\theta}\bigr) + \ln 2}{\E_{\theta^\star}\bigl[Z_{\theta}\bigr]} \right) .
\label{eq:posterior-DD-proof1}
\end{equation}
By the triangle inequality and the assumption $\ell\bigl(\theta^\star,\theta\bigr) \geq 2 \epsilon_n$ we can see that $\bigl\{\theta': \ell(\theta',\theta) > \epsilon_n \bigr\} \supseteq \bigl\{\theta': \ell(\theta',\theta^\star) < \epsilon_n \bigr\}$, so that
\[
\E_{\theta^\star}\bigl[Z_{\theta}\bigr] \geq \E_{\theta^\star}\Bigl[\P_{\pi}\bigl(\theta': \ell(\theta',\theta^\star) < \epsilon_n \big| X_{1:n} \bigr)\Bigr] \geq \frac{1}{c}
\]
by the uniform lower bound~\eqref{eq:posteriorlowerbound-DD-condition}. Substituting the above inequality into~\eqref{eq:posterior-DD-proof1} then yields
\begin{equation*}
\E_{\theta}\bigl[Z_{\theta}\bigr] \geq \exp\biggl( - c \Bigl( n \KL\bigl(P_{\theta^\star},P_{\theta}\bigr) + \ln 2\Bigr) \biggr) \,.
\end{equation*}
To conclude the proof, it suffices to take the supremum of the right-hand side over all $\theta^\star \in \Theta$ such that $\ell\bigl(\theta^\star,\theta\bigr) \geq 2 \epsilon_n$, and to identify the definition of $\psi\bigl(\epsilon_n,\theta,\ell\bigr)$.
\end{proof}

Note that, at first sight, our result may seem a little weaker than \citet[Theorem~2.1]{HoRoSH-15-AdaptivePosteriorConcentrationRates}, because we only define $\psi\bigl(\epsilon_n,\theta,\ell\bigr)$ in terms of $\KL$ instead of a general pre-metric $d$: in other words, we only consider the case $d(\theta,\theta') = \sqrt{\KL(P_{\theta'},P_{\theta})}$. However, it is still possible to derive a bound in terms of an arbitrary pre-metric $d$ by comparing $d$ and $\KL$ after applying Theorem~\ref{thm:lowerbound-posterior-distribDependent}.

In the case of the pre-metric $d(\theta,\theta') = \sqrt{\KL(P_{\theta'},P_{\theta})}$, we bypass
an additional technical assumption used for the
the similar lower bound of \citet[Theorem~2.1]{HoRoSH-15-AdaptivePosteriorConcentrationRates}; namely, that there exists a constant $C > 0$ such that
\[
\sup_{\theta,\theta'} \, P_{\theta'}^{\otimes n}\Bigl(\cL_n(\theta') - \cL_n(\theta) \geq C n \KL\bigl(P_{\theta'},P_{\theta}\bigr) \Bigl) \longrightarrow 0 \qquad \textrm{as } n \to +\infty\,,
\]
where the supremum is over all $\theta,\theta' \in \Theta$ satisfying $\psi\bigl(\epsilon_n,\theta,\ell\bigr) \leq \KL\bigl(P_{\theta'},P_{\theta}\bigr) \leq 2 \psi\bigl(\epsilon_n,\theta,\ell\bigr)$, and where $\cL_n(\theta) = \sum_{i=1}^n \ln \bigl(\d P_{\theta}/\d \mathfrak{m}\bigr)(X_i)$ denotes the log-likelihood function with respect to a common dominating measure~$\mathfrak{m}$.
Besides, we get an improved constant in the exponential in \eqref{eq:posteriorlowerbound-DD-conclusion}, with respect to \citet[Theorem~2.1]{HoRoSH-15-AdaptivePosteriorConcentrationRates}: by a factor of $3C/c$, which, since $C \geq 1$ in most cases,
is $3C/c \approx 3C \geq 3$ when $c \approx 1$. (A closer look at their proof can yield a constant arbitrarily close to $2C$, which is still larger than our $c$ by a factor of $2C/c \approx 2 C \geq 2$.)

\newpage
\section{References and comparison to the literature}
\label{sec:ref-beg}

We discuss in this section how novel (or not novel) our results and approaches are.
We first state where our main innovation lie in our eyes, and then discuss
the novelty or lack of novelty through a series of specific points.

\paragraph{Main innovations in a nutshell.}
We could find no reference indicating that the alternative distributions $\Q_i$ and $\Q_\theta$
could vary and do not need to be set to a fixed alternative $\Q_0$,
nor that arbitrary $[0,1]$--valued random variables $Z_i$ or $Z_\theta$ (i.e.,
not summing up to~$1$) could be considered.
These two elements are encompassed in the reduction~\eqref{eq:red5}, which is to be
considered our main new result.
The first application in Section~\ref{sec:applis} relies on
such arbitrary $[0,1]$--valued random variables $Z_\theta$ (but in the second application
the finitely many $Z_i$ sum up to~$1$).

That the sets $A_i$ considered in the reduction~\eqref{eq:red1}
form a partition of the underlying measurable space or that the finitely many random variables
$Z_i$ sum up to~1 (see \citealp{Gus-03-Fano}) were typical requirements in the literature until
recently, with one exception.
Indeed, \citet{ChenETAL-14-BayesRiskLowerBounds} noted in spirit that the requirement
of forming a partition was unnecessary, which we too had been aware of as early as~\citet{ExposeChevaleret},
where we also already mentioned the fact that in particular the alternative
distribution~$\Q$ had not to be fixed and could depend on $i$ or $\theta$.

\paragraph{Generalization to $f$--divergences (not a new result).}
\citet{Gus-03-Fano} generalized Fano-type inequalities with the Kullback-Leibler divergence
to arbitrary $f$--divergences, in the case where finitely many $[0,1]$--valued random variables $Z_1+\ldots+Z_N=1$ are considered;
see also \citet{ChenETAL-14-BayesRiskLowerBounds}.
Most of the literature focuses however on
Fano-type inequalities with the Kullback-Leibler divergence, like all references discussed below.

\paragraph{On the two-step methodology used (not a new result).}
The two-step methodology of Section~\ref{sec:prooftech}, which simply notes that
Bernoulli distributions are the main case to study when establishing
Fano-type inequalities, was well-known in the cases of
disjoint events or $[0,1]$--valued random variables summing up to~$1$.
This follows at various levels of clarity from
references that will be discussed in details in this section for other matters
(\citealp{HV94}, \citealp{Gus-03-Fano}, and \citealp{ChenETAL-14-BayesRiskLowerBounds})
and other references (\citealp[Section~D]{Zhang06}, and \citealp{HaVa-11-PairsfDivergences},
which is further discussed at the beginning of Section~\ref{sec:proofLBKL}).
In particular, the conjunction of a Bernoulli reduction and the use of
a lower bound on the $\kl$ function was already present in~\citet{HV94}.

Other, more information-theoretic statements and proof techniques of Fano's inequalities for
finitely many hypotheses as in Proposition~\ref{prop:FanoIT}
can be found, e.g., in
\citet[Theorem~2.11.1]{CoTh06}, \citet[Lemma~3]{Yu-97-AssouadFanoLeCam} or \citet[Chapter~VII, Lemma~1.1]{IbHa-81-StatisticalEstimation}
(they resort to classical formulas on the Shannon entropy, the conditional entropy,
and the mutual information).

\paragraph{On the reductions to Bernoulli distributions.}
Reduction~\eqref{eq:red5} is new at this level of generality, as we indicated, but all other reductions were known,
though sometimes proved in a more involved way.
Reduction~\eqref{eq:red1} and~\eqref{eq:red2} were already known and used by~\citet[Theorems~2, 7 and~8]{HV94}.
Reduction~\eqref{eq:red3} is stated in spirit by \citet{ChenETAL-14-BayesRiskLowerBounds} with
a constant alternative $\Q_\theta \equiv \Q$; see also a detailed discussion and comparison below
between their approach and the general approach we took in Section~\ref{sec:prooftech}.
We should also mention that \citet{DuWa-13-Fanocontinuum} provided preliminary (though more involved)
results towards the continuous reduction~\eqref{eq:red3}.
Finally, as already mentioned, a reduction with random variables like~\eqref{eq:red5}
was stated in a special case in~\citet{Gus-03-Fano}, for finitely many $[0,1]$--valued
random variables with $Z_1 + \ldots + Z_N = 1$.

\paragraph{On the lower bounds on the $\kl$ function (not really a new result).}
The inequalities~\eqref{eq:bd1} are folklore knowledge.
The first inequality in~\eqref{eq:bd2} can be found in~\citet{guntuboyina2011lower};
the second inequality is a new (immediate) consequence.
The inequalities~\eqref{eq:bd3} are a consequence, which we derived on our own,
of a refined Pinsker's inequality stated by \citet{OrWe-05-PinskerDistributionDependent}.

\paragraph{In-depth discussion of two articles.}
We now discuss two earlier contributions and indicate how our results encompass them:
the ``generalized Fanos's inequality'' of~\citet{ChenETAL-14-BayesRiskLowerBounds}
and the version of Fano's inequality by \citet{Birge05Fano}, which was designed
to also cover the case where $N=2$.

\subsection{On the ``generalized Fanos's inequality'' of~\citet{ChenETAL-14-BayesRiskLowerBounds}}

The Bayesian setting considered therein is the following;
it generalizes the setting of~\citet{HV94}, whose results we discuss
in a remark after the proof of Proposition~\ref{prop:Chenetal}.

A parameter space $(\Theta,\cG)$
is equipped with a prior probability measure $\nu$. A family of probability
distributions $(\P_\theta)_{\theta \in \Theta}$ on a measurable space $(\Omega,\cF)$,
some outcome space $(\cX,\cE)$, e.g., $\cX = \R^n$, and a random variable
$X : (\Omega,\cF) \to (\cX,\cE)$ are considered. We denote by $\E_\theta$
the expectation under $\P_\theta$. Of course we may have
$(\Omega,\cF) = (\cX,\cE)$ and $X$ be the identity, in which case $\P_\theta$
will be the law of $X$ under $\P_\theta$.

The goal is either to estimate $\theta$ or to take good actions:
we consider a measurable target space $(\cA,\cH)$, that may or may not be equal to $\Theta$.
The quality of a prediction or of an action is measured by a measurable loss function
$L : \Theta \times \cA \to [0,1]$.
The random variable $X$
is our observation, based on which we construct a $\sigma(X)$--measurable
random variable $\hat{a}$ with values in $\cA$.
Putting aside all measurability issues (here and in the rest of this subsection),
the risk of $\hat{a}$
in this model equals
\[
R\bigl(\hat{a}\bigr) = \bigintsss_{\Theta} \, \E_\theta\Bigl[ L\bigl(\theta,\hat{a}\bigr) \Bigr]
\d\nu(\theta)
\]
and the Bayes risk in this model is the smallest such possible risk,
\[
\Rb = \inf_{\hat{a}} \, R\bigl(\hat{a}\bigr)\,,
\]
where the infimum is over all $\sigma(X)$--measurable random variables
with values in~$\cA$.

\citet{ChenETAL-14-BayesRiskLowerBounds}
call their main result (Corollary~5) a ``generalized Fano's inequality'';
we state it and prove it below not only for $\{0,1\}$--valued loss functions $L$ as in the original article,
but for any $[0,1]$--valued loss function. The reason behind this extension is that we not only have the reduction~\eqref{eq:red3}
with events, but we also have the reduction~\eqref{eq:red5} with $[0,1]$--valued random variables.
We also feel that our proof technique is more direct and more natural.

We only deal with Kullback-Leibler divergences, but
the result and proof below readily extend to $f$--divergences.

\begin{proposition}
\label{prop:Chenetal}
In the setting described above, the Bayes risk is always larger than
\[
\Rb \geq 1 + \frac{\displaystyle{\left( \inf_{\Q} \int_\Theta \KL(\P_\theta,\Q) \d \nu(\theta) \right)}
+ \displaystyle{\ln\!\left( 1 + \inf_{a \in \cA} \int_\Theta L(\theta,a)\d\nu(\theta) \right)}}{
\displaystyle{\ln\!\left( 1 - \inf_{a \in \cA} \int_\Theta L(\theta,a)\d\nu(\theta) \right)}}\,,
\]
where the infimum in the numerator is over all probability measures $\Q$ on $(\Omega,\cF)$.
\end{proposition}

\begin{proof}
We fix $\hat{a}$ and an alternative $\Q$.
The combination of~\eqref{eq:red5} and~\eqref{eq:bd2}, with $Z_\theta = 1 - L\bigl(\theta,\hat{a}\bigr)$,
yields
\begin{equation}
\label{eq:proofChen1}
1 - \bigintsss_{\Theta} \, \E_\theta\Bigl[ L\bigl(\theta,\hat{a}\bigr) \Bigr]
\d\nu(\theta) \leq
\frac{\displaystyle{\int_\Theta \KL(\P_\theta,\Q) \d \nu(\theta)} + \ln\bigl(2-q_{\hat{a}}\bigr)}{\ln\bigl(1/q_{\hat{a}}\bigr)}\,,
\end{equation}
where $\E_{\Q}$ denotes the expectation with respect to $\Q$ and
\[
q_{\hat{a}} = 1 - \bigintsss_{\Theta} \, \E_{\Q}\Bigl[ L\bigl(\theta,\hat{a}\bigr) \Bigr]
\d\nu(\theta)\,.
\]
As $q \mapsto 1/\ln(1/q)$ and $q \mapsto \ln(2-q)/\ln(1/q)$ are both increasing,
taking the supremum over the $\sigma(X)$--measurable random variables $\hat{a}$
in both sides of~\eqref{eq:proofChen1} gives
\begin{equation}
\label{eq:proofChen2}
1 - \Rb \leq
\frac{\displaystyle{\int_\Theta \KL(\P_\theta,\Q) \d \nu(\theta)} + \ln\bigl(2-q^\star\bigr)}{\ln\bigl(1/q^{\star}\bigr)}
\end{equation}
where
\begin{equation}
\label{eq:proofChen3}
q^\star = \sup_{\hat{a}} q_{\hat{a}} =
1 - \inf_{\hat{a}} \bigintsss_{\Theta} \, \E_{\Q}\Bigl[ L\bigl(\theta,\hat{a}\bigr) \Bigr]
\d\nu(\theta)
= 1 - \inf_{a \in \cA} \int_\Theta L(\theta,a)\d\nu(\theta)\,,
\end{equation}
as is proved below. Taking the infimum of the right-hand side of~\eqref{eq:proofChen2}
over $\Q$ and rearranging concludes the proof.

It only remains to prove the last inequality of~\eqref{eq:proofChen3} and actually,
as constant elements $a \in \cA$ are special cases of random variables $\hat{a}$,
we only need to prove that
\begin{equation}
\label{eq:proofChen4}
\inf_{\hat{a}} \bigintsss_{\Theta} \, \E_{\Q}\Bigl[ L\bigl(\theta,\hat{a}\bigr) \Bigr] \d\nu(\theta)
\geq \inf_{a \in \cA} \int_\Theta L(\theta,a)\d\nu(\theta)\,.
\end{equation}
Now, each $\hat{a}$ that is $\sigma(X)$--measurable can be rewritten
$\hat{a} = \overline{a}(X)$ for some measurable function $\overline{a} : \cX \to \cA$; then,
by the Fubini-Tonelli theorem:
\[
\bigintsss_{\Theta} \, \E_{\Q}\Bigl[ L\bigl(\theta,\hat{a}\bigr) \Bigr] \d\nu(\theta)
= \bigintss_{\cX} \left( \bigintsss_\Theta L\bigl(\theta,\overline{a}(x)\bigr) \d\nu(\theta) \right)
\!\d\Q(x)
\geq \bigintss_{\cX} \left( \inf_{a \in \cA} \bigintsss_\Theta L\bigl(\theta,a\bigr) \d\nu(\theta) \right)
\!\d\Q(x)\,,
\]
which proves~\eqref{eq:proofChen4}.
\end{proof}

\begin{remark}
{\em
As mentioned by \citet{ChenETAL-14-BayesRiskLowerBounds},
one of the major results of \citet{HV94}, namely, their Theorem~8, is a special case of
Proposition~\ref{prop:Chenetal}, with $\Theta = \cA$
and the loss function $L(\theta,\theta') = \ind{\{\theta \ne \theta' \}}$.
The (opposite of the) denominator in the lower bound on the Bayes risk then takes the simple form
\[
- \ln \! \left( 1 - \inf_{\theta' \in \Theta} \int_\Theta L(\theta,\theta')\d\nu(\theta) \right)
= - \ln \! \left( \sup_{\theta \in \Theta} \nu \bigl( \{\theta\} \bigr) \!\right) \defeq H_\infty(\nu)\,,
\]
which is called the infinite-order R{\'e}nyi entropy of the probability distribution~$\nu$.
\citet{HV94} only dealt with the case of discrete sets $\Theta$ but the extension to
continuous $\Theta$ is immediate, as we showed in Section~\ref{sec:prooftech}.
}
\end{remark}

\subsection{Comparison to~\citet{Birge05Fano}: \\ An interpolation between Pinsker's and Fano's inequalities}

The most classical version of Fano's inequality, that is, the right-most side of~\eqref{eq:partition-classic+ln2q} below, is quite 
impractical for small values of $N$ (cf.\ \citealp{Birge05Fano}), and even useless when $N = 2$, the latter case being straightforward to 
deal with by several well-known tools, for example, by Pinsker's inequality. 
One of the main motivations of \citet{Birge05Fano} was therefore to get an inequality that
would be useful for all $N \geq 2$. 
His inequality is stated next; it
only deals with events $A_1,\ldots,A_N$
forming a partition of the underlying measurable space.
As should be clear from its proof this assumption is crucial. (See Appendix~\ref{sec:Ext} for a pointer to an extended version of this article where
a proof following the methodology described in Section~\ref{sec:prooftech} is provided.)

\begin{theorem}[Birg{\'e}'s lemma]
\label{th:Birge}
Given an underlying measurable space $(\Omega,\cF)$,
for all $N \geq 2$,
for all probability distributions $\P_1,\ldots,\P_N$,
for all events $A_1,\ldots,A_N$ forming a partition of $\Omega$,
\[
\min_{1 \leq i \leq N} \P_i(A_i) \leq \max\!\left\{ c_N, \,\,
\frac{\ol{K}}{\ln(N)}
\right\} \qquad \mbox{where} \qquad \ol{K} = \frac{1}{N-1} \sum_{i=2}^N \KL(\P_i,\P_1)
\]
and where $(c_N)_{N \geq 2}$ is a decreasing sequence, where
each term $c_N$ is defined as the unique $c \in (0,1)$ such that
\begin{equation}
\label{def:cNBirge}
\frac{- \bigl( c \ln(c) + (1-c)\ln(1-c) \bigr)}{c} + \ln(1-c) = \ln \!\left( \frac{N-1}{N} \right).
\end{equation}
We have, for instance, $c_2 \approx 0.7587$ and $c_3 \approx 0.7127$, while
$\lim c_N = 0.63987$.
\end{theorem}

However, a first drawback of the bound above lies in the $\ol{K}$
term: one cannot pick a convenient $\Q$ as in the
bounds~\eqref{eq:partition-classic+ln2q}--\eqref{eq:partition-Pinsker-fort} below.
A second drawback is that the result is about the minimum of the $\P_i(A_i)$, not about their average.
In contrast,
the versions of Fano's inequality based the $\kl$ lower bounds~\eqref{eq:bd2}, \eqref{eq:bd1}, and \eqref{eq:bd3}
respectively lead to the following inequalities, stated
in the setting of Theorem~\ref{th:Birge} and by picking constant alternatives $\Q$:
\begin{align}
\label{eq:partition-classic+ln2q}
\frac{1}{N} \sum_{i=1}^N \P_i(A_i) \leq &
\frac{\displaystyle{\frac{1}{N} \inf_{\Q} \sum_{i=1}^N \KL(\P_i,\Q)} + \ln\!\left(2-\frac{1}{N}\right)}{\ln(N)}
\leq
\frac{\displaystyle{\frac{1}{N} \inf_{\Q} \sum_{i=1}^N \KL(\P_i,\Q)} + \ln(2)}{\ln(N)}\,, \\
\label{eq:partition-Pinsker-fort}
\mbox{and} \qquad \frac{1}{N} \sum_{i=1}^N \P_i(A_i) \leq & \frac{1}{N} +
\sqrt{\frac{\displaystyle{\frac{1}{N} \inf_{\Q} \sum_{i=1}^N \KL(\P_i,\Q)}}{\max\bigl\{\ln(N),\,2\bigr\}}}\;.
\end{align}
The middle term in~\eqref{eq:partition-classic+ln2q} was derived---with a different formulation---by \citet{ChenETAL-14-BayesRiskLowerBounds}, see Proposition~\ref{prop:Chenetal} above.

\paragraph{Discussion.}
We note that unlike the right-most side of~\eqref{eq:partition-classic+ln2q},
both the middle term in~\eqref{eq:partition-classic+ln2q}
and the bound~\eqref{eq:partition-Pinsker-fort}
yield useful bounds for all $N \geq 2$, and in particular, for $N=2$.
Even better, \eqref{eq:partition-Pinsker-fort} implies both Pinsker's inequality and, lower bounding the maximum by $\ln(N)$, a bound as useful as Theorem~\ref{th:Birge} or
Proposition~\ref{prop:FanoIT} in case of a partition.
Indeed, in practice, the additional additive $1/N$ term and the additional square root do not prevent from obtaining the desired lower
bounds, as illustrated in Section~\ref{sec:usecase-sequential}.

Therefore, our inequality~\eqref{eq:partition-Pinsker-fort} provides some interpolation
between Pinsker's and Fano's inequalities: it simultaneously deals with all values $N \geq 2$.

\newpage
\section{Proofs of the lower bounds on $\kl$ stated in Section~\ref{sec:LBdiv} \\ (and proof of an improved Bretagnolle-Huber inequality)}
\label{sec:proofLBKL}

We prove in this section
the convexity inequalities~\eqref{eq:bd2}
and~\eqref{eq:bd2bis} as well as the
refined Pinsker's inequality and its consequence~\eqref{eq:bd3}.
Using the same techniques and methodology as for establishing these bounds,
we also improve in passing the Bretagnolle-Huber inequality.

The main advantage of the Bernoulli reductions of Section~\ref{sec:avg}
is that we could then capitalize in Section~\ref{sec:excombi} (and also in
Section~\ref{sec:otherapplications}) on any lower bound on the Kullback-Leibler
divergence $\kl(p,q)$ between Bernoulli distributions. In the same spirit, our key argument below to prove the refined Pinsker's inequality and the Bretagnolle-Huber inequality (which hold for arbitrary probability distributions) is in both cases an inequality between the Kullback-Leibler divergence and the total variation distance between Bernoulli distributions. This simple but deep observation was made in great generality by \citet{HaVa-11-PairsfDivergences}.

\subsection{Proofs of the convexity inequalities~\eqref{eq:bd2} and~\eqref{eq:bd2bis}}
\label{sec:proofLBKL-conv}

\begin{proof}
Inequality~\eqref{eq:bd2bis} follows from~\eqref{eq:bd2}
by noting that the function $q \in (0,1) \mapsto \ln(2-q)\big/\ln(1/q)$
is dominated by $q \in (0,1) \mapsto 0.21 + 0.79\,q$.

Now, the shortest proof of~\eqref{eq:bd2} notes that the
duality formula for the Kullback-Leibler divergence between Bernoulli distributions---already
used in~\eqref{eq:CramerChernoff-upper}---ensures that,
for all $p \in [0,1]$ and $q \in (0,1]$,
\[
\kl(p,q) = \sup_{\lambda \in \R} \biggl\{ \lambda p - \ln \Bigl( q \bigl(\e^\lambda-1\bigr) + 1 \Bigr) \biggr\}
\geq p \ln \biggl( \frac{1}{q} \biggr) - \ln(2-q)
\]
for the choice $\lambda = \ln(1/q)$.
\end{proof}

An alternative, longer but more elementary proof
uses a direct convexity argument, as in \citet[Example~II.4]{guntuboyina2011lower},
which already included the inequality of interest in the special case
when $q = 1/N$; see also \citet{ChenETAL-14-BayesRiskLowerBounds}.
We deal separately with $p=0$ and $p=1$, and thus restrict our attention
to $p \in (0,1)$ in the sequel.
For $q \in (0,1)$, as $p \mapsto \kl(p,q)$ is convex and differentiable
on~$(0,1)$, we have
\begin{equation}
\label{eq:cvx:ln2-q}
\forall \, (p,p_0) \in (0,1)^2, \qquad \quad
\kl(p,q) - \kl(p_0,q) \geq
\underbrace{\ln\!\left(\frac{p_0 (1-q)}{(1-p_0) q}\right)}_{\frac{\partial}{\partial p}\kl(p_0,q)}
(p-p_0)\,.
\end{equation}
The choice $p_0 = 1/(2-q)$ is such that
\[
\frac{p_0}{1-p_0} = \frac{1}{1-q}\,,
\qquad \mbox{thus} \qquad
\ln\!\left(\frac{p_0 (1-q)}{(1-p_0) q}\right) = \ln \biggl( \frac{1}{q} \biggr)\,,
\]
and
\[
\kl(p_0,q) = \frac{1}{2-q} \ln \biggl( \frac{1/(2-q)}{q} \biggr)
+ \frac{1-q}{2-q} \ln \biggl( \frac{(1-q)/(2-q)}{1-q} \biggr)
= \frac{1}{2-q} \ln \biggl( \frac{1}{q} \biggr)
+ \ln \biggl( \frac{1}{2-q} \biggr)\,.
\]
Inequality~\eqref{eq:cvx:ln2-q} becomes
\[
\forall \, p \in (0,1), \qquad \quad
\kl(p,q) -
\frac{1}{2-q} \ln \biggl( \frac{1}{q} \biggr)
+ \ln(2-q)
\geq \left( p - \frac{1}{2-q} \right) \ln \biggl( \frac{1}{q} \biggr)\,,
\]
which proves as well the bound~\eqref{eq:bd2}.

\subsection{Proofs of the refined Pinsker's inequality and of its consequence~\eqref{eq:bd3}}
\label{sec:proofLBKL-Pinsker}

The next theorem is a stronger version of Pinsker's inequality for Bernoulli distributions,
that was proved\footnote{We also refer the reader to~\citet[Lemma~1]{KeSa-98uai-LargeDeviationMethods} and \citet[Theorem~3.2]{BeKo-13-ConcentrationMissingMass} for dual inequalities upper bounding the moment-generating function of the Bernoulli distributions.} by \citet{OrWe-05-PinskerDistributionDependent}.
Indeed, note that the function $\phi$ defined below satisfies $\min \phi = 2$, so that the next
theorem always yields an improvement over the most classical version of Pinsker's inequality: $\kl(p,q) \geq 2 (p-q)^2$.

We provide below an alternative elementary proof for Bernoulli distributions of this refined Pinsker's inequality.
The extension to the case of general distributions, via the contraction-of-entropy property, is stated at the end of this section.

\begin{theorem}[A refined Pinsker's inequality by \citet{OrWe-05-PinskerDistributionDependent}]
\label{th:Pinskerrefined}
For all $p,q \in [0,1]$,
\[
\kl(p,q) \geq \frac{\ln\bigl((1-q)/q\bigr)}{1-2q} \, (p-q)^2
\defeq \varphi(q) \, (p-q)^2 \,,
\]
where the multiplicative factor $\varphi(q) = (1-2q)^{-1} \ln\bigl((1-q)/q\bigr)$ is defined for all $q \in [0,1]$ by extending it by continuity as $\varphi(1/2) = 2$ and $\varphi(0) = \varphi(1) = +\infty$.
\end{theorem}

The proof shows that $\varphi(q)$ is the optimal multiplicative factor in front of $(p-q)^2$
when the bounds needs to hold for all $p \in [0,1]$; the proof also provides a natural explanation
for the value of $\varphi$. \\

\begin{proof}
The stated inequality is satisfied for $q \in \{0,1\}$ as $\kl(p,q) = +\infty$ in these cases unless $p=q$. The special case $q=1/2$ is addressed at the end of the proof.
We thus fix $q \in (0,1)\setminus \{1/2\}$ and set $f(p) = \kl(p,q) / (p-q)^2$ for $p \ne q$, with a continuity extension at $p = q$. We exactly show that $f$ attains its
minimum at $p=1-q$, from which the result (and its optimality) follow by noting that
\[
f(1-q) = \frac{\kl(1-q,q)}{(1-2q)^2} = \frac{\ln\bigl((1-q)/q\bigr)}{1-2q} = \varphi(q)\,.
\]
Given the form of $f$, it is natural to perform a second-order Taylor expansion of $\kl(p,q)$
around $q$. We have
\begin{equation}
\label{eq:diffkl}
\frac{\partial}{\partial p}\kl(p,q) = \ln\!\left(\frac{p (1-q)}{(1-p) q}\right)
\qquad \mbox{and} \qquad
\frac{\partial^2}{\partial^2 p}\kl(p,q) = \frac{1}{p (1-p)}\defeq \psi(p)\,,
\end{equation}
so that Taylor's formula with integral remainder reveals that for $p \ne q$,
\[
f(p) = \frac{\kl(p,q)}{(p-q)^2} = \frac{1}{(p-q)^2} \int_q^p \frac{\psi(t)}{1!} (p-t)^1 \d t =
\int_0^1 \psi\bigl(q + u (p-q)\bigr) (1-u) \d u\,.
\]
This rewriting of $f$ shows that $f$ is strictly convex (as $\psi$ is so).
Its global minimum is achieved at the unique point where its derivative vanishes.
But by differentiating under the integral sign, we have, at $p=1-q$,
\[
f'(1-q) = \int_0^1 \psi'\bigl(q + u (1-2q)\bigr) \, u (1-u) \d u = 0\,;
\]
the equality to~$0$
follows from the fact that the function $u \mapsto \psi'\bigl(q + u (1-2q)\bigr) u (1-u)$ is antisymmetric around $u=1/2$ (essentially because $\psi'$ is antisymmetric itself around $1/2$).
As a consequence, the convex function $f$ attains its global minimum at $1-q$, which concludes the proof
for the case where $q \in (0,1)\setminus \{1/2\}$.

It only remains to deal with $q=1/2$: we use the continuity of $\kl(p,\,\cdot\,)$ and $\varphi$
to extend the obtained inequality from $q \in [0,1]\setminus \{1/2\}$ to $q = 1/2$.
\end{proof}

We now prove the second inequality of~\eqref{eq:bd3}.
A picture is helpful, see Figure~\ref{fig:varphi}.

\begin{figure}[t]
\begin{center}
\begin{tabular}{ccc}
\includegraphics[width=0.3\textwidth]{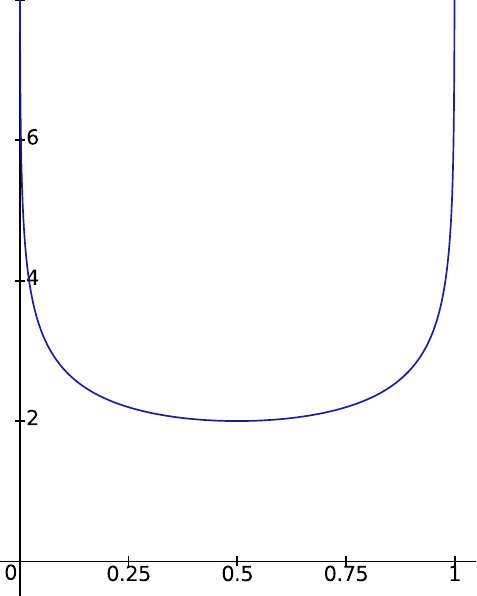} & \phantom{je pousse} &
\includegraphics[width=0.3\textwidth]{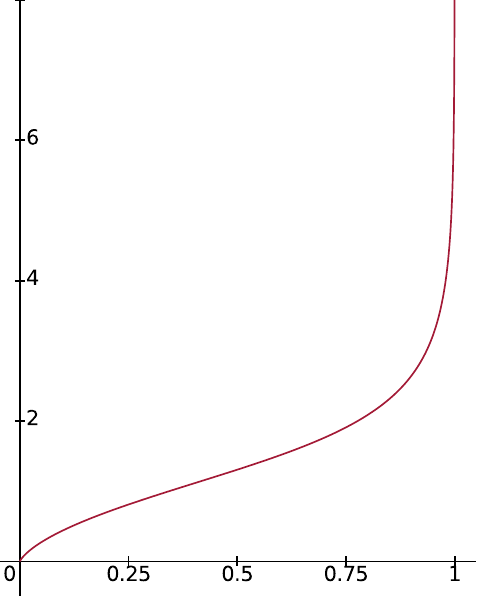}
\end{tabular}
\end{center}
\caption{\label{fig:varphi} Plots of $\varphi$ [left]
and $x \in (0,1) \mapsto \varphi(x)-\ln(1/x)$ [right].}
\end{figure}

\begin{corollary}
\label{cor:FanoPinskerFort}
For all $q \in (0,1]$, we have $\varphi(q) \geq 2$ and $\varphi(q) \geq \ln(1/q)$. Thus,
for all $p \in [0,1]$ and $q \in (0,1)$,
\[
p \leq q + \sqrt{\frac{\kl(p,q)}{\max\bigl\{\ln(1/q),\,2\bigr\}}}\,.
\]
\end{corollary}

Slightly sharper bounds are possible, like $\varphi(q) \geq (1+q)(1+q^2)\ln(1/q)$
or $\varphi(q) \geq \ln(1/q) + 2.5\,q$, but we were unable to exploit these
refinements in our applications.

\paragraph{General refined Pinsker's inequality.}
The following result, which improves on Pinsker's inequality, is due to \citet{OrWe-05-PinskerDistributionDependent}. Our approach through Bernoulli distributions enables to derive it in an elementary (and enlightening) way: by combining Theorem~\ref{th:Pinskerrefined} and the data-processing inequality (Lemma~\ref{lem:dataProcessingIneq--fdiv}).

\begin{theorem}
Let $\P$ and $\Q$ be two probability distributions on the same measurable space $(\Omega,\cF)$.
Then
\[
\forall \, A \in \cF, \qquad \bigl| \P(A) - \Q(A) \bigr|
\leq \sqrt{\frac{\KL(\P,\Q)}{\varphi\bigl(\Q(A)\bigr)}}\,,
\]
where $\varphi \geq 2$ is defined in the statement of Theorem~\ref{th:Pinskerrefined}.
In particular, the total variation distance between $\P$ and $\Q$
is bounded as
\[
\sup_{A \in \cF} \, \bigl| \P(A) - \Q(A) \bigr|
\leq \sqrt{\frac{\KL(\P,\Q)}{\displaystyle{\inf_{A \in \cF} \, \varphi\bigl(\Q(A)\bigr)}}}\,.
\]
\end{theorem}

\subsection{An improved Bretagnolle-Huber inequality}
\label{sec:proofLBKL-BH}

The Bretagnolle-Huber inequality was introduced by \citet{BrHu-78-RisqueMinimax,BrHu-79-RisqueMinimax}. The multiplicative factor $\e^{-1/\e} \geq 0.69$ in our statement~\eqref{eq:bd4} below is a slight improvement over the original $1/2$ factor. For all $p,q \in [0,1]$,
\begin{equation}
\label{eq:bd4}
1 -|p-q| \geq \e^{-1/\e} \, \e^{-\kl(p,q)}\,,
\qquad \mbox{thus} \qquad
q \geq p - 1 + \e^{-1/\e} \, \e^{-\kl(p,q)}\,.
\end{equation}
It is worth to note that
\citet{BrHu-78-RisqueMinimax} also proved the inequality
\[
|p-q| \leq \sqrt{1-\exp\bigl(-\kl(p,q)\bigr)}\,,
\]
which improves as well upon the Bretagnolle-Huber inequality with the $1/2$ factor, but which is neither better nor worse than~\eqref{eq:bd4}. \smallskip

Now, via the data-processing inequality (Lemma~\ref{lem:dataProcessingIneq--fdiv}),
we get from~\eqref{eq:bd4}
\[
1 - \sup_{A \in \cF} \, \bigl| \P(A) - \Q(A) \bigr| \geq \e^{-1/\e} \, \e^{-\KL(\P,\Q)}\,.
\]
The left-hand side can be rewritten as $\inf_{A \in \cF} \bigl\{\P(A) + \Q(A^c)\bigr\}$, where $A^c$ denotes the complement of~$A$. Therefore, the above inequality is a lower bound on the test affinity between $\P$ and~$\Q$.
For the sake of comparison to~\eqref{eq:Fano1}, we can restate the general
version of the Bretagnolle-Huber inequality as:
for all $A \in \cF$,
\begin{equation}
\label{eq:bd4bis}
\Q(A) \geq \P(A) - 1 + \e^{-1/\e} \, \e^{-\KL(\P,\Q)}\,.
\end{equation}

We now provide a proof of~\eqref{eq:bd4}; note that our improvement was made possible because we reduced the proof to very elementary arguments in the case of Bernoulli distributions. \\

\begin{proof}
The case where $p \in \{0,1\}$ or $q \in \{0,1\}$ can be handled separately; we consider $(p,q)  \in (0,1)^2$ in the sequel.
The derivative of the function $x \in (0,1) \mapsto x \ln \bigl(x/(1-q)\bigr)$ equals $1+\ln(x)-\ln(1-q)$, so that the function achieves its minimum at $x = (1-q)/\e$, with value $-(1-q)/\e \geq -1/\e$. Therefore,
\[
- \kl(p,q)
= - p \ln \!\left( \frac{p}{q} \right) - (1-p) \ln \!\left( \frac{1-p}{1-q} \right)
\leq - p \ln \!\left( \frac{p}{q} \right) + \frac{1}{\e}
= p \Biggl( \ln \!\left( \frac{q}{p} \right) + \frac{1}{\e} \Biggr)
+ (1-p) \frac{1}{\e}\,.
\]
Therefore, using the convexity of the exponential,
\[
\e^{-\kl(p,q)} \leq p \, \exp\Biggl( \ln \!\left( \frac{q}{p} \right) + \frac{1}{\e} \Biggr)
+ (1-p) \, \e^{1/\e} = \bigl(q+(1-p)\bigr) \, \e^{1/\e}\,,
\]
which shows that
\[
1 - (p-q) \geq \e^{-1/\e} \, \e^{-\kl(p,q)}\,.
\]
By replacing $q$ by $1-q$ and $p$ by $1-p$, we also get
\[
1 - (q-p) = 1 - \bigl( (1-p) - (1-q) \bigr) \geq \e^{-1/\e} \, \e^{-\kl(1-p,1-q)} = \e^{-1/\e} \, \e^{-\kl(p,q)}\,.
\]
This concludes the proof, as $1 - |p-q|$ is equal to the smallest value
between $1 - (p-q)$ and $1 - (q-p)$.
\end{proof}

\newpage
\bibliographystyle{plainnat}
\bibliography{referencesFano}

\newpage
\appendix

\section{On the sharpness of Fano-type inequalities of Section~\ref{sec:prooftech}}
\label{sec:sharpness}

The reductions of Section~\ref{sec:avg} are sharp in the sense that they can hold with equality
(they cannot be improved at this level of generality).

For the Kullback-Leibler divergence, they lead to inequalities of the form
$\kl\bigl(\ol{p},\ol{q}\bigr) \leq \ol{K}$. We are interested in upper bounds on $\ol{p}$.
We introduce the generalized inverse of $\kl$ in its second argument: for all $q \in [0,1]$
and all $y \geq 0$,
\[
\kl(\,\cdot\,,q)^{(-1)}(y) \defeq \sup\bigl\{p \in [0,1]: \ \kl(p,q) \leq y\bigr\}\,;
\]
when $q \in (0,1)$,
it is thus equal to the largest root $q$ of the equation $\kl(p,q) = y$ if $y \leq \ln(1/q)$ or to $1$ otherwise.
From $\kl\bigl(\ol{p},\ol{q}\bigr) \leq \ol{K}$ the best general upper bound on $\ol{p}$ is
\[
\ol{p} \leq \kl\bigl(\,\cdot\,,\ol{q}\bigr)^{(-1)}\bigl( \ol{K} \bigr)\,.
\]
This formulation should be reminiscent of \citet[Theorem~2]{Birge05Fano},
but has one major practical drawback: it is unreadable, and this is why we considered
the lower bounds of Section~\ref{sec:LBdiv}.

Question is now how sharp these lower bounds on $\kl$ are.
Bounds~\eqref{eq:bd1} and~\eqref{eq:bd2} are of the form
\[
p \leq \frac{\kl(p,q)}{\ln(1/q)} + \varepsilon(q) \,,
\]
where the $\varepsilon(q)$ quantity vanishes when $q \to 0$.
Now, in the applications, $q$ is typically small and the main term $\kl(p,q)/\ln(1/q)$ is of the order of a constant.
Therefore, the lemma below explains that up to the $\varepsilon$ quantity, the bounds~\eqref{eq:bd1} and~\eqref{eq:bd2}
of Section~\ref{sec:LBdiv} are essentially optimal.

The bound~\eqref{eq:bd3} therein is of the form
\[
p \leq \sqrt{\frac{\kl(p,q)}{\ln(1/q)}} + \varepsilon(q) \,,
\]
but given the discussion above, it can also be considered optimal in spirit, as in the applications
$q$ is typically small and the main term $\kl(p,q)/\ln(1/q)$ is of the order of a constant.

\begin{lemma}
For all $q \in (0,1)$ and $p \in [0,1]$,
whenever $p \geq q$, we have
\[
\kl(p,q) \leq p \ln \!\left(\frac{1}{q}\right)
\qquad \mbox{thus} \qquad
p \geq \frac{\kl(p,q)}{\ln(1/q)}\,.
\]
\end{lemma}

\begin{proof}
We note that
when $p \geq q$, we have $(1-p)/(1-q) \leq 1$, so that
\[
\kl(p,q) = p \ln \!\left(\frac{1}{q}\right)
+ \underbrace{p \ln(p)}_{\leq 0}
+ (1-p) \underbrace{\ln\!\left(\frac{1-p}{1-q}\right)}_{\leq 0}
\leq p \ln \!\left(\frac{1}{q}\right),
\]
hence the first inequality.
\end{proof}

\newpage
\section{From Bayesian posteriors to point estimators}
\label{sec:bayesianToPoint}

We recall below a well-known result that indicates how to construct good point estimators from good Bayesian posteriors (Section~\ref{sec:conversion} below). One theoretical benefit is that this result can be used to convert known minimax lower bounds for point estimation into minimax lower bounds for posterior concentration rates (Section~\ref{sec:conversion-appli} below). This technique is thus a---less direct---alternative to the method we presented in Section~\ref{sec:bayesianPosteriorRates}.

\subsection{The conversion}
\label{sec:conversion}
The following statement is a nonasymptotic variant of Theorem~2.5 by \citet{GhosalETAL-00-ConvergenceRatesPosterior} (see also Chapter~12, Proposition~3 by \citealp{LeCam-86-Asymptotic_Methods_in_Statistical_Decision_Theory}, as well as Section~5.1 by \citealp{HoRoSH-15-AdaptivePosteriorConcentrationRates}). We consider the same setting as in Section~\ref{sec:bayesianPosteriorRates} and assume in particular that the underlying probability measure
is given by $\P_\theta = P_\theta^{\otimes n}$, that is, that $(X_1,\ldots,X_n)$ is the identity random variable.

\begin{proposition}[From Bayesian posteriors to point estimators]
\label{prop:lowerbound-posterior-conversion}
\ \\
Let $n \geq 1$, $\delta > 0$, and $\theta \in \Theta$. Let $\hat{\theta}_n=\hat{\theta}_n(X_1,\ldots,X_n)$ be any estimator satisfying, $\P_{\theta}$--almost surely,
\begin{equation}
\P_{\pi}\Bigl(\theta': \ell(\theta',\hat{\theta}_n) < \epsilon_n \, \big| \, X_{1:n} \Bigr) \geq \sup_{\tilde{\theta} \in \Theta} \P_{\pi}\Bigl(\theta': \ell\bigl(\theta',\tilde{\theta} \,\bigr) < \epsilon_n \, \big| \, X_{1:n} \Bigr) - \delta \,.
\label{eq:posterior-conversion-defthetahat}
\end{equation}
Then,
\begin{equation}
\P_{\theta}\biggl(\P_{\pi}\bigl(\theta': \ell(\theta',\theta) \geq \epsilon_n \, \big| \, X_{1:n} \bigr) \geq \frac{1-\delta}{2}\biggr) \geq \P_{\theta}\Bigl(\ell\bigl(\hat{\theta}_n,\theta\bigr) \geq 2 \epsilon_n \Bigr)~.
\label{eq:posterior-conversion-conclusion}
\end{equation}
\end{proposition}

\noindent
This result implies that if $\hat{\theta}_n$ is a center of a ball that almost maximizes the posterior mass---see  assumption~\eqref{eq:posterior-conversion-defthetahat}---and if the posterior mass concentrates around $\theta$ at a rate $\epsilon'_n < \epsilon_n$---so that the left-hand side of~\eqref{eq:posterior-conversion-conclusion} vanishes by Markov's inequality---then $\hat{\theta}_n$ is $(2\epsilon_n)$--close to $\theta$ with high probability.
Therefore, at least from a theoretical viewpoint, a good posterior distribution can be converted into a good point estimator,
by defining $\hat{\theta}_n$ based on $\P_{\pi}(\,\cdot\,\,|\,X_{1:n})$ such that~\eqref{eq:posterior-conversion-defthetahat} holds, i.e., by taking an approximate argument of the supremum.
A measurable such $\hat{\theta}_n$ exists as soon as $\Theta$ is a separable topological space and the function $\tilde{\theta} \mapsto \P_{\pi}\bigl(\theta': \ell(\theta',\tilde{\theta} \,) < \epsilon_n \, \big| \, x_{1:n} \bigr)$ is lower-semicontinuous for $\mathfrak{m}^{\otimes n}$-almost every $x_{1:n} \in \cX^n$ (see the end of the proof of Corollary~\ref{cor:minimax2BayesLB} for more details).
\bigskip

\begin{proof}
Denote by $B_{\ell}(\theta,\epsilon) \defeq \{\theta' \in \Theta: \, \ell(\theta',\theta) < \epsilon\}$ the open $\ell$--ball of center $\theta$ and radius~$\epsilon$. By the triangle inequality we have the following inclusions of events:
\begin{align}
\Bigl\{\ell\bigl(\hat{\theta}_n,\theta\bigr) \geq 2 \epsilon_n \Bigr\} & \subseteq \Bigl\{ B_{\ell}\bigl(\hat{\theta}_n,\epsilon_n\bigr) \cap B_{\ell}(\theta,\epsilon_n) = \varnothing \Bigr\} \nonumber \\
& \subseteq \Bigl\{\P_{\pi}\bigl(B_{\ell}\bigl(\hat{\theta}_n,\epsilon_n\bigr) \, \big| \, X_{1:n} \bigr) + \P_{\pi}\bigl(B_{\ell}(\theta,\epsilon_n) \, \big| \, X_{1:n} \bigr) \leq 1 \Bigr\} \nonumber \\
& \subseteq \biggl\{\P_{\pi}\bigl(B_{\ell}(\theta,\epsilon_n) \, \big| \, X_{1:n} \bigr) \leq \frac{1+\delta}{2} \biggr\} \label{eq:posterior-conversion-1} \\
& = \biggl\{1-\P_{\pi}\bigl(\theta': \ell(\theta',\theta)<\epsilon_n \, \big| \, X_{1:n} \bigr) \geq \frac{1-\delta}{2} \biggr\} \\
& = \biggl\{\P_{\pi}\bigl(\theta': \ell(\theta',\theta) \geq \epsilon_n \, \big| \, X_{1:n} \bigr) \geq \frac{1-\delta}{2} \biggr\} \,, \nonumber
\end{align}
where \eqref{eq:posterior-conversion-1} follows from the lower bound $\P_{\pi}\bigl(B_{\ell}\bigl(\hat{\theta}_n,\epsilon_n\bigr) \, \big| \, X_{1:n} \bigr) \geq \P_{\pi}\bigl(B_{\ell}(\theta,\epsilon_n) \, \big| \, X_{1:n} \bigr) - \delta$, which holds by assumption~\eqref{eq:posterior-conversion-defthetahat} on~$\hat{\theta}_n$. This concludes the proof.
\end{proof}

\subsection{Application to posterior concentration lower bounds}
\label{sec:conversion-appli}

We explained above that a good posterior distribution can be converted into a good point estimator. As noted by \citet{GhosalETAL-00-ConvergenceRatesPosterior} this conversion can be used the other way around:
if we have a lower bound on the minimax rate of estimation, then Proposition~\ref{prop:lowerbound-posterior-conversion} provides a lower bound on the minimax posterior concentration rate, as formalized in the following corollary. Assumption~\eqref{eq:inprob-lowerbound} below corresponds to an in-probability minimax lower bound.

\begin{corollary}
\label{cor:minimax2BayesLB}
Let $n \geq 1$. Consider the setting of Section~\ref{sec:bayesianPosteriorRates}, with underlying probability measure~$\P_\theta = P_\theta^{\otimes n}$ when the unknown parameter is $\theta$. Assume that $\Theta$ is a separable topological space and that $\tilde{\theta} \mapsto \ell\bigl(\theta',\tilde{\theta} \,\bigr)$ is continuous for all $\theta' \in \Theta$.
Assume also that for some absolute constant $c < 1$, we have
\begin{equation}
\inf_{\hat{\theta}_n \,\, \mbox{\tiny \rm est.}} \,\, \sup_{\theta \in \Theta} \, \P_{\theta}\Bigl(\ell\bigl(\hat{\theta}_n,\theta\bigr) \geq 2 \epsilon_n \Bigr) \geq 1-c \,,
\label{eq:inprob-lowerbound}
\end{equation}
where the infimum is taken over all estimators $\hat{\theta}_n$.
Then, for all priors $\pi'$ on $\Theta$,
\begin{equation}
\inf_{\theta \in \Theta} \, \E_{\theta}\Big[\P_{\pi'}\bigl(\theta': \ell(\theta',\theta) < \epsilon_n \, \big| \, X_{1:n} \bigr) \Bigr] \leq \frac{1+c}{2} < 1 \,.
\label{eq:minimax2BayesLB-result}
\end{equation}
\end{corollary}

\begin{proof}
Let $\delta > 0$ be a parameter that we will later take arbitrarily small. Fix any estimator $\hat{\theta}_n$ satisfying~\eqref{eq:posterior-conversion-defthetahat} for the prior $\pi'$, i.e., that almost maximizes the posterior mass on an open ball of radius~$\epsilon_n$. (See the end of the proof for details on why such a measurable $\hat{\theta}_n$ exists.)
Then, Proposition~\ref{prop:lowerbound-posterior-conversion} used for all $\theta \in \Theta$ entails that
\[
\sup_{\theta \in \Theta} \, \P_{\theta}\biggl(\P_{\pi'}\bigl(\theta': \ell(\theta',\theta) \geq \epsilon_n \, \big| \, X_{1:n} \bigr) \geq \frac{1-\delta}{2}\biggr) \geq \sup_{\theta \in \Theta} \, \P_{\theta}\Bigl(\ell\bigl(\hat{\theta}_n,\theta\bigr) \geq 2 \epsilon_n \Bigr) \geq 1-c \,,
\]
where the last inequality follows from the assumption~\eqref{eq:inprob-lowerbound}. Now we use Markov's inequality to upper bound the left-hand side above and obtain
\[
\frac{2}{1-\delta} \, \sup_{\theta \in \Theta} \, \E_{\theta}\Big[\P_{\pi'}\bigl(\theta': \ell(\theta',\theta) \geq \epsilon_n \, \big| \, X_{1:n} \bigr)\Bigr]
\geq
\sup_{\theta \in \Theta} \, \P_{\theta}\biggl(\P_{\pi'}\bigl(\theta': \ell(\theta',\theta) \geq \epsilon_n \, \big| \, X_{1:n} \bigr) \geq \frac{1-\delta}{2}\biggr)
\geq 1-c \,.
\]
Letting $\delta \to 0$ and dividing both sides by $2$ yields
\[
1 - \inf_{\theta \in \Theta} \E_{\theta}\Big[\P_{\pi'}\bigl(\theta': \ell(\theta',\theta) < \epsilon_n \, \big| \, X_{1:n} \bigr) \Bigr] \geq \frac{1-c}{2}\,.
\]
Rearranging terms concludes the proof of~\eqref{eq:minimax2BayesLB-result}. We now address the technical issue mentioned at the beginning of the proof. \smallskip

\noindent
\emph{Why a measurable $\hat{\theta}_n$ exists.}~~Note that it is possible to choose $\hat{\theta}_n$ satisfying~\eqref{eq:posterior-conversion-defthetahat} with $\pi'$ in a mesurable way as soon as $\Theta$ is a separable topological space and
\[
\psi: \tilde{\theta} \in \Theta \longmapsto \P_{\pi'}\bigl(\theta': \ell(\theta',\tilde{\theta} \,) < \epsilon_n \, \big| \, x_{1:n} \bigr)
 \]
is lower-semicontinuous for $\mathfrak{m}^{\otimes n}$--almost every $x_{1:n} \in \cX^n$, and thus $\P_{\theta}$--almost surely for all $\theta \in \Theta$. The reason is that, in that case, it is possible to equate the supremum of $\psi$ over $\Theta$ to a supremum on a countable subset of~$\Theta$. Next, and thanks to the continuity assumption on $\ell$, we prove that the desired lower-semicontinuity holds true for all $x_{1:n} \in \cX^n$ (not just almost all of them).

To that end, we show the lower-semicontinuity at any fixed $\theta^\star \in \Theta$.
Consider any sequence $(\tilde{\theta}_i)_{i \geq 1}$ in $\Theta$ converging to $\theta^\star$.
For all $x_{1:n} \in \cX^n$,
by Fatou's lemma applied to the well-defined probability distribution $\P_{\pi'}(\,\cdot\,\,|\,x_{1:n})$,
we have,
\begin{align}
\liminf_{i \to +\infty} \ \P_{\pi'}\Bigl(\theta': \ell\bigl(\theta',\tilde{\theta}_i \,\bigr) < \epsilon_n \, \big| \, x_{1:n} \Bigr) & = \liminf_{i \to +\infty} \ \E_{\pi'}\!\left[\indicator{\{\ell(\theta',\tilde{\theta}_i \,) < \epsilon_n\}} \, \big| \, x_{1:n} \, \right] \nonumber \\
& \geq \E_{\pi'}\!\biggl[ \, \underbrace{\liminf_{i \to +\infty} \indicator{\{\ell(\theta',\tilde{\theta}_i \,) < \epsilon_n\}}}_{\textrm{$=1$ if $\ell(\theta', \theta^\star) < \epsilon_n$}} \, \Big| \, x_{1:n} \, \biggr] \label{eq:liminf-LB} \\
& \geq \P_{\pi'}\Bigl(\theta': \ell\bigl(\theta',\theta^\star \,\bigr) < \epsilon_n \, \big| \, x_{1:n} \Bigr) \,, \nonumber
\end{align}
where in~\eqref{eq:liminf-LB} we identify that the $\liminf$ equals $1$ as soon as
$\ell(\theta', \theta^\star) < \epsilon_n$ by continuity of $\tilde{\theta} \mapsto \ell\bigl(\theta',\tilde{\theta}\,\bigr)$ at $\tilde{\theta} = \theta^\star$.
\end{proof}

\newpage
\section{On Jensen's inequality}
\label{sec:Jensen}

Classical statements of Jensen's inequality for convex functions $\phi$ on $C \subseteq \R^n$ either assume that the underlying probability measure is supported on a finite number of points or that the convex subset $C$ is open. In the first case, the proof follows directly from the definition of convexity, while in the second case, it is a consequence of the existence of subgradients. In both cases, it is assumed that the function $\phi$ under consideration only takes finite values. In this article, Jensen's inequality is applied several times to non-open convex sets $C$,
like $C = [0,1]^2$ or $C = [0,+\infty)$ and/or convex functions $\varphi$ that can possibly be equal
to $+\infty$ at some points.

The restriction of $C$ being open is easy to drop when the dimension equals $n=1$, i.e., when $C$ is an interval; it was dropped, e.g., by \citet[pages 74--76]{Fer-76-MathematicalStatistics} in higher dimensions,
thanks to a proof by induction to address possible boundary effects with respect to the arbitrary convex set~$C$.
Let $\cB(\R^n)$ denote the Borel $\sigma$--field of~$\R^n$.

\begin{lemma}[Jensen's inequality for general convex sets; \citealp{Fer-76-MathematicalStatistics}]
\label{lem:Jensen}
Let $C \subseteq \R^n$ be any non-empty convex Borel subset of~$\R^n$ and $\phi:C \to \R$ be any convex Borel function. Then, for all probability measures $\mu$ on $\bigl(\R^n,\cB(\R^n)\bigr)$ such that $\mu(C) = 1$ and $\int \Vert x \Vert \d \mu(x) < + \infty$, we have
\begin{equation}
\label{eq:Jensen-show}
\int x \d \mu(x) \in C \qquad \textrm{and} \qquad \phi\!\left(\int x \d \mu(x)\right) \leq \int_C \phi(x) \d \mu(x) \,,
\end{equation}
where the integral of $\varphi$ against $\mu$ is well-defined in $\R \cup \{+\infty\}$.
\end{lemma}

Our contribution is the following natural extension.

\begin{lemma}
\label{lem:Jensen2}
The result of Lemma~\ref{lem:Jensen} also holds for
any convex Borel function $\phi:C \to \R \cup \{+\infty\}$.
\end{lemma}

We rephrase this extension in terms of random variables.
Let $C \subseteq \R^n$ be any non-empty convex Borel subset of $\R^n$ and $\phi:C \to \R \cup \{+\infty\}$ be any convex Borel function.
Let $X$ be an integrable random variable from any probability space $(\Omega,\cF,\P)$ to $\bigl(\R^n,\cB(\R^n)\bigr)$, such that $\P(X \in C)=1$. Then
\begin{equation*}
\E[X] \in C \qquad \textrm{and} \qquad \phi\bigl(\E[X]\bigr) \leq \E\bigl[\phi(X)\bigr] \,,
\end{equation*}
where $\E\bigl[\phi(X)\bigr]$ is well-defined in $\R \cup \{+\infty\}$. \\

\begin{proof}
We first check that $\phi_- = \max\{-\varphi,\,0\}$ is $\mu$--integrable on $C$,
so that the integral of $\varphi$ against $\mu$ is well-defined in $\R \cup \{+\infty\}$.
To that end, we will prove that $\phi$ is lower bounded on $C$ by an affine function:
$\varphi(x) \geq a^{\transp} x + b$ for all $x \in C$, where $(a,b) \in \R^2$, from which it follows that
$\phi_-(x) \leq \Vert a \Vert \Vert x \Vert + \Vert b \Vert$ for all $x \in C$ and thus
\[
\int_C \phi_-(x) \d \mu(x) \leq \int_C \Bigl( \Vert a \Vert \Vert x \Vert + \Vert b \Vert \Bigr) \d \mu(x) = \Vert a \Vert \int_C \Vert x \Vert \d \mu(x) + \Vert b \Vert < + \infty \,.
\]
So, it only remains to prove the affine lower bound.
If the domain $\{\phi < +\infty\}$ is empty, any affine function is suitable.
Otherwise, $\{\phi < +\infty\}$ is a non-empty convex set, so that its relative interior $R$ is also non-empty (see \citealp[Theorem~6.2]{Roc-72-ConvexAnalysis}); we fix $x_0 \in R$. But, by \citet[Theorem~23.4]{Roc-72-ConvexAnalysis},
the function $\phi$ admits a subgradient at $x_0$, that is, there exists $a \in \R^n$ such that $\phi(x) \geq \phi(x_0) + a^T (x-x_0)$ for all $x \in C$. This concludes the first part of this proof.

In the second part, we show the inequality~\eqref{eq:Jensen-show}
via a reduction to the case of real-valued functions. Indeed, note that if $\mu(\phi=+\infty) > 0$ then the desired inequality is immediate. We can thus assume that $\mu(\phi<+\infty) = 1$. But, using Lemma~\ref{lem:Jensen} with the non-empty convex Borel subset $\tilde{C} = \{\phi<+\infty\}$ and the real-valued convex Borel function $\tilde{\phi}:\tilde{C} \to \R$ defined by $\tilde{\phi}(x)=\phi(x)$, we get, since $\mu(\tilde{C}) = 1$:
\[
\int x \d \mu(x) \in \tilde{C} \qquad \textrm{and} \qquad \tilde{\phi}\!\left(\int x \d \mu(x)\right) \leq \int_{\tilde{C}} \tilde{\phi}(x) \d \mu(x) \,.
\]
Using the facts that $\tilde{\phi}(x) = \phi(x)$ for all $x \in \tilde{C}$ and that $\mu\bigl(C \setminus \tilde{C}\bigr) = 1 - 1 = 0$ entails~\eqref{eq:Jensen-show}.
\end{proof}

We now complete our extension by tacking the conditional form of Jensen's inequality.

\begin{lemma}[A general conditional Jensen's inequality]
\label{lem:Jensencond}
Let $C \subseteq \R^n$ be any non-empty convex Borel subset of $\R^n$ and $\phi:C \to \R \cup \{+\infty\}$ be any convex Borel function.
Let $X$ be an integrable random variable from any probability space $(\Omega,\cF,\P)$ to $\bigl(\R^n,\cB(\R^n)\bigr)$, such that $\P(X \in C)=1$. Then, for every sub-$\sigma$--field $\cG$ of $\cF$, we have, $\P$--almost surely,
\begin{equation*}
\E[X \, | \, \cG ] \in C \qquad \textrm{and} \qquad \phi\bigl(\E[X \, | \, \cG ]\bigr) \leq \E\bigl[\phi(X) \,|\, \cG \bigr] \,,
\end{equation*}
where $\E\bigl[\phi(X) \,|\, \cG \bigr]$ is $\P$--almost-surely well-defined in $\R \cup \{+\infty\}$.
\end{lemma}

\begin{proof}
The proof follows directly from the unconditional Jensen's inequality (Lemma~\ref{lem:Jensen2} above) and from the existence of regular conditional distributions.
More precisely, by \citet[Theorems~2.1.15 and~5.1.9]{Durrett-ProbTheory-4thEd} applied to the case where $(S,\mathcal{S}) = \bigl(\R^n,\cB(\R^n)\bigr)$, there exists a regular conditional distribution of $X$ given $\cG$. That is, there exists a function $K:\Omega \times \cB(\R^n) \to [0,1]$ such that:
\begin{itemize}
	\item[(P1)] for every $B \in \cB(\R^n)$, \ $\omega \in \Omega \mapsto K(\omega,B)$ is $\cG$--measurable and $\P(X \in B \,\big|\, \cG) = K(\,\cdot\,,B)$ $\P$--a.s.;
	\item[(P2)] for $\P$--almost all $\omega \in \Omega$, the mapping $B \mapsto K(\omega,B)$ is a probability measure on $\bigl(\R^n,\cB(\R^n)\bigr)$.
\end{itemize} \smallskip
Moreover, as a consequence of~(P1),
\begin{itemize}
	\item[(P1')] for every Borel function $g : \R^n \to \R$ such that $g(X)$ is $\P$--integrable
or such that $g$ is nonnegative,
\[
\int g(x) \, K(\,\cdot\,,\d x) = \E \bigl[ g(X) \, \big| \, \cG \bigr] \qquad \P\mbox{--a.s.}
\]
\end{itemize}
\noindent Now, given our assumptions and thanks to~(P1) and~(P1'):
\begin{itemize}
\item[(P3)] by $\P(X \in C)=1$ we also have $K(\,\cdot\,,C) = \P(X \in C\,|\,\cG) = 1$ $\P$--a.s.;
\item[(P4)] since $X$ is $\P$--integrable, so is $\int \Vert x \Vert \, K(\,\cdot\,,\d x) = \E\bigl[ \Vert X \Vert  \, \big| \, \cG \bigr]$, which is therefore $\P$--a.s.\ finite.
\end{itemize} \smallskip

We apply Lemma~\ref{lem:Jensen2} with the probability measures $\mu_\omega = K(\omega,\,\cdot\,)$,
for those $\omega$ for which the properties stated in (P2), (P3) and~(P4) actually hold;
these $\omega$ are $\P$--almost all elements of $\Omega$. We get, for these $\omega$,
\[
\int x \, K(\omega,\d x) \in C \qquad \textrm{and} \qquad \phi\!\left(\int x \, K(\omega,\d x)\right) \leq \int_C \phi(x) \, K(\omega,\d x) \,,
\]
where the integral in the right-hand side is well defined in $\R \cup \{+\infty\}$.
Thanks to~(P1'), and by decomposing $\varphi(X)$ into $\varphi_{-}(X)$, which is integrable
(see the beginning of the proof of Lemma~\ref{lem:Jensen2}), and $\varphi_+(X)$, which
is nonnegative, we thus have proved that $\P$--a.s.,
\[
\E\bigl[X \, \big| \, \cG \bigr] \in C \qquad \textrm{and} \qquad \phi\Bigl(\E\bigl[X \, \big| \, \cG \bigr]\Bigr) \leq \E\bigl[\phi(X) \big| \cG \bigr] \,,
\]
which concludes the proof.
\end{proof}

\newpage
\section{Extended version of this article}
\label{sec:Ext}

Here ends the version of this article published in \emph{Statistical Science}. 
Supplementary material follows, featuring the following additional appendices.

\paragraph{Appendix~\ref{sec:divproofs}:} Proofs of the data-processing inequality (Lemma~\ref{lem:dataProcessingIneq--fdiv}) \\
\phantom{\textbf{Appendix~\ref{sec:divproofs}:}~~~}and of the joint convexity of $\Div$ (Corollary~\ref{cor:jointconvKL--fdiv})

\paragraph{Appendix~\ref{sec:Fanotype}:} Additional material on the Fano-type inequalities
of Section~\ref{sec:prooftech}, namely
\begin{itemize}
\item A sharper lower bound on $\div$ for the Hellinger distance
\item Finding a good constant alternative $\Q$
\end{itemize}

\paragraph{Appendix~\ref{sec:Birge-HAL}:} On Birg{\'e}'s lemma, namely
\begin{itemize}
\item A proof of Theorem~\ref{th:Birge}
\item Two other statements of this lemma (the original one and a simplification of it)
\end{itemize}

\newpage
\input{Fano-Appendix-HAL}

\end{document}

%% file: Fano-Appendix-HAL.tex
\newpage

\

\vfill

\begin{center}
{\LARGE Supplementary material for the article \vspace{.25cm} \\
``\papertitle'' \vspace{.25cm} \\
by Gerchinovitz, M{\'e}nard, Stoltz}
\end{center}

\vfill
\vfill

\newpage
\section{Proofs of the data-processing inequality (Lemma~\ref{lem:dataProcessingIneq--fdiv}) \\
and of the joint convexity of $\Div$ (Corollary~\ref{cor:jointconvKL--fdiv})}
\label{sec:divproofs}

As indicated in the main body of the article,
the proof of Lemma~\ref{lem:dataProcessingIneq--fdiv}
is extracted from \citet[Section 4.2]{AlSi-66-fdivergences}, see
also \citet[Proposition~1.2]{ThInfo-Pardo}. Note that it can be refined:
\citet[Lemmas 7.5 and 7.6]{Gray} establishes~\eqref{eq:RNderivative-images} below
and then derives some (stronger) data-processing equality (not inequality). \\

\begin{proofref}{Lemma~\ref{lem:dataProcessingIneq--fdiv}, data-processing inequality}
We recall that $\E_{\Q}$ denotes the expectation with respect to a measure~$\Q$.
Let~$X$ be a random variable from $(\Omega,\cF)$ to $(\Omega',\cF')$.
We write the Lebesgue decomposition~\eqref{eq:Lbgdec} of $\P$ with respect to $\Q$.

We first show that $(\Pac)^X \ll \Q^X$ and that the Radon-Nikodym derivative
of $(\Pac)^X$ with respect to $\Q^X$ equals
\begin{equation}
\frac{\d (\Pac)^X}{\d \Q^X} = \E_{\Q}\!\left[\frac{\d \Pac}{\d \Q} \, \Big| X = \cdot \right] \defeq \gamma \,;
\label{eq:RNderivative-images}
\end{equation}
i.e., $\gamma$ is any measurable function such that $\Q$--almost surely,
$\E_{\Q}\bigl[(\d \Pac / \d \Q) \, \big| X \bigr]  = \gamma(X)$.
Indeed, using that $\Pac \ll \Q$, we have, for all $A \in \cF'$,
\begin{align}
(\Pac)^X(A) & = \Pac(X \in A) = \bigintsss_{\Omega} \indicator{A}(X) \, \frac{\d \Pac}{\d \Q} \d \Q
= \bigintsss_{\Omega} \indicator{A}(X) \,\, \E_{\Q}\!\left[\frac{\d \Pac}{\d \Q} \, \Big| X \right] \! \d \Q \label{eq:RNderivative-towerrule} \\
& = \int_{\Omega} \indicator{A}(X) \, \gamma(X) \d \Q = \int_{\Omega'} \indicator{A} \, \gamma \d \Q^X \, , \nonumber
\end{align}
where the last equality in~\eqref{eq:RNderivative-towerrule} follows by the tower rule.

Second, by unicity of the Lebesgue decomposition,
the decomposition of $\P^X$ with respect to $\Q^X$ is therefore given by
\begin{align*}
\P^X = \bigl( \P^X \bigr)_{\ac} + \bigl( \P^X \bigr)_{\sing}
\qquad \mbox{where} & \qquad
\bigl( \P^X \bigr)_{\ac} = (\Pac)^X + \bigl( \Psing \bigr)^X_{\ac} \\
\qquad \mbox{and} & \qquad
\bigl( \P^X \bigr)_{\sing} = \bigl( \Psing \bigr)^X_{\sing}\,.
\end{align*}
The inner $\ac$ and $\sing$ symbols refer to the pair $\P,\Q$ while the outer
$\ac$ and $\sing$ symbols refer to $\P^X,\Q^X$.

{\allowdisplaybreaks
We use this decomposition for the first equality below and integrate~\eqref{eq:Mf} for the first inequality below:
\begin{align}
\Div\bigl(\P^X,\Q^X\bigr)
& = \bigintsss_{\Omega'} f\Biggl( \frac{\d (\Pac)^X}{\d \Q^X} + \frac{\d \bigl( \Psing \bigr)^X_{\ac}}{\d \Q^X} \Biggr)
\d \Q^X + \bigl( \Psing \bigr)^X_{\sing}(\Omega') \, M_f \nonumber \\
& \leq \bigintsss_{\Omega'} f\Biggl( \frac{\d (\Pac)^X}{\d \Q^X} \Biggr)
\d \Q^X + \Bigl( \bigl( \Psing \bigr)^X_{\ac}(\Omega') + \bigl( \Psing \bigr)^X_{\sing}(\Omega') \Bigl)
\, M_f \nonumber \\
& = \int_{\Omega'} f(\gamma) \d \Q^X + (\Psing)^X(\Omega') \, M_f \nonumber \\
& = \int_{\Omega} f\bigl(\gamma(X)\bigr) \d \Q + \Psing(\Omega) \, M_f \nonumber \\
& = \bigintsss_{\Omega} f\Biggl(\E_{\Q}\!\left[\frac{\d \P}{\d \Q} \, \Big| X \right] \Biggr) \d \Q
+ \Psing(\Omega) \, M_f\nonumber \\
& \leq \bigintsss_{\Omega} \E_{\Q}\Biggl[ f\biggl(\frac{\d \P}{\d \Q}\biggr) \, \bigg| X \Biggr] \d \Q
+ \Psing(\Omega) \, M_f \label{eq:RNderivative-condJensen}  \\
& = \bigintsss_{\Omega} f\biggl(\frac{\d \P}{\d \Q}\biggr)\! \d \Q  +
\Psing(\Omega) \, M_f = \Div(\P,\Q) \, , \nonumber
\end{align}
where the inequality in~\eqref{eq:RNderivative-condJensen} is a
consequence of the conditional Jensen's inequality in its general form stated in Appendix~\ref{sec:Jensen}, Lemma~\ref{lem:Jensencond}, with $\phi=f$ and $C=[0,+\infty)$,
and where the final equality follows from the tower rule.
}
\end{proofref}

The joint convexity of $\Div$ (Corollary~\ref{cor:jointconvKL--fdiv}) may be proved directly, in two steps.
First, the log-sum inequality is generalized into the fact that the mapping
$(p,q) \in [0,+\infty)^2 \mapsto q\,f(p/q)$ is jointly convex. Second,
a common dominating measure like $\mu = \P_1 + \P_2 + \Q_1 + \Q_2$
is introduced, Radon-Nikodym derivatives $p_j$ and $q_j$ are introduced
for the $\P_j$ and $\Q_j$ with respect to $\mu$, and the generalized log-sum
inequality is applied pointwise.

We suggest to see instead Corollary~\ref{cor:jointconvKL--fdiv}
as an elementary consequence of the data-processing inequality.
\bigskip

\begin{proofref}{Corollary~\ref{cor:jointconvKL--fdiv}, joint convexity of $\Div$}
We augment the probability space $\Omega$ into $\Omega' = \{1,2\} \times \Omega$, which we
equip with the $\sigma$--algebra $\cF'$ generated by the events $A \times B$,
where $A \in \bigl\{ \emptyset, \, \{1\}, \, \{2\}, \, \{1,2\} \bigr\}$ and $B \in \cF$.
We define the random pair $(J,X)$ on this space by the projections
\[
X : (j,\omega) \in \{1,2\} \times \Omega \longmapsto \omega
\qquad \mbox{and} \qquad
J : (j,\omega) \in \{1,2\} \times \Omega \longmapsto j\,,
\]
and denote by $\P$ the joint distribution of the random pair $(J,X)$ such that $J \sim 1+\Ber(\lambda)$
and $X|J \sim \P_J$. More formally, $\P$ is the unique probability distribution on $(\Omega',\cF')$ such that, for all $(j,B) \in \{1,2\} \times \cF$,
\[
\P\bigl(\{j\} \times B \bigr) =
\left((1-\lambda)  \indicator{\{j=1\}} + \lambda \indicator{\{j=2\}} \right) \P_j(B)\,.
\]
Similarly we define the joint probability distribution $\Q$ on $(\Omega',\cF')$ using the conditional distributions $\Q_1$ and $\Q_2$ instead of $\P_1$ and $\P_2$.

The corollary follows directly from the data-processing inequality $\Div\!\left(\P^X,\Q^X\right) \leq \Div(\P,\Q)$, as the laws of $X$ under $\P$ and
$\Q$ are respectively given by
\[
\P^X = (1-\lambda) \P_1 + \lambda \P_2
\qquad \mbox{and} \qquad
\Q^X = (1-\lambda)  \Q_1 + \lambda \Q_2\,,
\]
while elementary calculations show that $\Div(\P,\Q) = (1-\lambda) \, \Div(\P_1,\Q_1) + \lambda \, \Div(\P_2,\Q_2)$.

Indeed, for the latter point, we consider the Lebesgue decompositions of $\P_j$ with respect to
$\Q_j$, where $j \in \{1,2\}$:
\[
\P_j = \P_{j,\ac} + \P_{j,\sing}\,, \qquad \mbox{where} \qquad
\P_{j,\ac} \ll \Q_j \quad \mbox{and} \quad \P_{j,\sing} \bot \Q_j \,.
\]
The (unique) Lebesgue decomposition of $\P = \Pac + \Psing$ with respect to $\Q$ is
then given by
\[
\frac{\d \Pac}{\d \Q}(j,\omega) = \ind{\{j=1\}}\,\frac{\d \P_{1,\ac}}{\d \Q_1}(\omega)
+ \ind{\{j=2\}}\,\frac{\d \P_{2,\ac}}{\d \Q_2}(\omega)
\]
and for all $(j,B) \in \{1,2\} \times \cF$,
\[
\Psing \bigl(\{j\} \times B \bigr) =
\left((1-\lambda)  \indicator{\{j=1\}} + \lambda \indicator{\{j=2\}} \right) \P_{j,\sing}(B)\,.
\]
This entails that
\begin{align*}
\Div(\P,\Q) & =
\bigintsss_{\{1,2\} \times \Omega} f \! \left( \frac{\d \Pac}{\d \Q}(j,\omega) \right)\! \d \Q(j,\omega)
+ \Psing \bigl(\{1,2\} \times \Omega \bigr) \, M_f
\\
& =
(1-\lambda)  \bigintsss_{\Omega} f \! \left( \frac{\d \Pac}{\d \Q}(1,\omega) \right)\! \d \Q_1(\omega)
+ \lambda \bigintsss_{\Omega} f \! \left( \frac{\d \Pac}{\d \Q}(2,\omega) \right)\! \d \Q_2(\omega) \\
& \phantom{\lambda \bigintsss_{\Omega} f \! \left( \frac{\d \Pac}{\d \Q}(1,\omega) \right)\! \d \Q_1(\omega)}
+ \bigl( (1-\lambda)  \, \P_{1,\sing}(\Omega) + \lambda \, \P_{2,\sing}(\Omega) \bigr) \, M_f \\
& = (1-\lambda)  \, \Div(\P_1,\Q_1) + \lambda \, \Div(\P_2,\Q_2)\,.
\end{align*}
\vspace{-1.3cm}
\

\end{proofref}

\newpage
\section{Additional material on the Fano-type inequalities of Section~\ref{sec:prooftech}}
\label{sec:Fanotype}

We first provide (Section~\ref{sec:Hellsharper}) a sharper
bound than the bound~\eqref{eq:LB-H}
exhibited in Section~\ref{sec:LBdiv} for the case of the
Hellinger distance and which read
\[
p \leq q + \sqrt{1-\bigl(1-h^2(p,q)/2\bigr)^2} = q + \sqrt{h^2(p,q)\bigl(1-h^2(p,q)/4\bigr)}\,.
\]

We then (Section~\ref{sec:cstaltQ}) study quantities of the form
\[
\inf_{\Q} \sum_{i=1}^N \Div(\P_i,\Q)\,,
\]
that arise in the bounds of Section~\ref{sec:excombi}
in the case of partitions. We discuss
what $\Q$ should be chosen and what bounds can be achieved.

\subsection{A sharper lower bound on $\div$ for the Hellinger distance}
\label{sec:Hellsharper}

We follow and slightly generalize \citet[Example~II.6]{guntuboyina2011lower}.
As we prove below,
we get the bound
\begin{equation}
\label{eq:Helltoprove}
p \leq q + (1-2q) \, h^2(p,q) \bigl(1-h^2(p,q)/4 \bigr) +
2 \sqrt{q(1-q)} \, \bigr(1-h^2(p,q)/2\bigl) \sqrt{h^2(p,q) \bigl(1-h^2(p,q)/4 \bigr)}\,.
\end{equation}
It can be seen that this bound is a general expression of the bound stated by \citet[Example~II.6]{guntuboyina2011lower}.
This bound is slightly tighter than~\eqref{eq:LB-H}, by construction (as we solve exactly an equation and perform no bounding)
but it is much less readable. It anyway leads to similar conclusions in practice. \\

\begin{proof}
Assuming that $\underline{h}^2 = h^2(p,q)$ is given and fixed,
we consider the equation, for the unknown $x \in [0,1]$,
\[
\underline{h}^2 = 2 \biggl( 1 - \Bigl( \sqrt{q} \sqrt{x} + \sqrt{1-q} \, \sqrt{1-x} \Bigr) \biggr) \,;
\]
this equation is satisfied for $x = p$, by definition of $h^2(p,q)$. Rearranging it, we get the
equivalent equation
\[
(1-x)(1-q) = \bigl( 1 - \underline{h}^2/2 - \sqrt{q} \sqrt{x} \bigr)^2
= \bigl( 1 - \underline{h}^2/2 \bigr)^2 - 2 \bigl( 1 - \underline{h}^2/2 \bigr) \sqrt{q} \, \sqrt{x}
+ q x\,,
\]
or equivalently again,
\[
x - 2 \bigl( 1 - \underline{h}^2/2 \bigr) \sqrt{q} \, \sqrt{x}
+ \bigl( 1 - \underline{h}^2/2 \bigr)^2 - 1 + q = 0\,.
\]
Solving this second-order equation for $\sqrt{x}$, we see that
all solutions $\sqrt{x}$, including $\sqrt{p}$, are
smaller than the largest root; in particular,
\[
\sqrt{p} \leq \bigl( 1 - \underline{h}^2/2 \bigr) \sqrt{q} +
\underbrace{\sqrt{\bigl( 1 - \underline{h}^2/2 \bigr)^2 q - \bigl( 1 - \underline{h}^2/2 \bigr)^2 + 1 - q}}_{=
\sqrt{(1-q) \, \underline{h}^2 (1-\underline{h}^2/4)}}\,.
\]
Put differently,
\begin{align*}
p & \leq \bigl( 1 - \underline{h}^2/2 \bigr)^2 q + (1-q) \, \underline{h}^2 \bigl(1-\underline{h}^2/4 \bigr)
+ 2 \sqrt{q(1-q)} \, \bigr(1-\underline{h}^2/2\bigl) \sqrt{\underline{h}^2 \bigl(1-\underline{h}^2/4 \bigr)} \\
& = q + (1-2q) \, h^2(p,q) \bigl(1-h^2(p,q)/4 \bigr) +
2 \sqrt{q(1-q)} \, \bigr(1-h^2(p,q)/2\bigl) \sqrt{h^2(p,q) \bigl(1-h^2(p,q)/4 \bigr)}\,,
\end{align*}
which was the expression to obtain.
\end{proof}

\subsection{Finding a good constant alternative $\Q$}
\label{sec:cstaltQ}

The key term in the bounds of Section~\ref{sec:excombi} in the case
of a partition or of random variables summing up to~$1$
is given by
\[
\inf_{\Q} \sum_{i=1}^N \Div(\P_i,\Q)\,.
\]
We however need a closed-form (or at least, a more concrete) expression
of this quantity for these bounds to have a practical interest.
This is the issue we tackle in this section.

Instead of simply studying quantities of the form indicated above,
we consider
\[
\inf_{\Q} \sum_{i=1}^N \alpha_i \, \Div(\P_i,\Q)
\]
where $\alpha = (\alpha_1,\ldots,\alpha_N)$, with all $\alpha_i > 0$,
denotes some convex combination.

Sometimes calculations are easy in practice for some specific $\Q$,
as we illustrated, for instance, in Section~\ref{sec:usecase-sequential}.
Otherwise, the lemma below indicates a good candidate, given by the
weighted average $\ol{\P}_\alpha$ of the distributions $\P_i$.

To appreciate its performance, we denote by
\[
B_f(\alpha) = \max_{j=1,\ldots,N} \Div(\delta_j,\alpha)
\]
the maximal $f$--divergence between a Dirac mass $\delta_j$ at $j$
and the convex combination $\alpha$.
This bound equals $\ln\bigl(1/\min\{\alpha_1,\ldots,\alpha_N\}\bigr)$
for a Kullback-Leibler divergence and $1/\min\{\alpha_1,\ldots,\alpha_N\} -1$ for the $\chi^2$--divergence.

\begin{lemma}
Let $\P_1,\ldots,\P_N$ be $N$ probability distributions on the same measurable space $(\Omega,\cF)$
and let $\alpha = (\alpha_1,\ldots,\alpha_N)$ be a convex combination made of positive weights. Then,
\[
\inf_{\Q} \sum_{i=1}^N \alpha_i \, \Div\!\left(\P_i,\Q\right)
\leq \sum_{i=1}^N \alpha_i \, \Div\!\left(\P_i,\ol{\P}_\alpha \right) \leq B_f(\alpha)\,,
\] \vspace{-.3cm}

\noindent where the infimun is over all probability distributions $\Q$ on $(\Omega,\cF)$
and where $\displaystyle{\ol{\P}_\alpha \defeq \sum_{i=1}^N \alpha_i \, \P_i}$.
\end{lemma}

The first inequality holds with equality in the case
of the Kullback-Leibler divergence,
as follows from the so-called compensation equality
(see, e.g., \citealp{YaBa-99-MinimaxRates}
or \citealp[Example~II.4]{guntuboyina2011lower}):
assuming with no loss of generality in this case (since $M_f = +\infty$)
that $\P_j \ll \Q$ for all $j \in \{1,\ldots,N\}$, we have $\ol{\P}_\alpha \ll \Q$ and $\mathrm{d} \P_j/\mathrm{d} \Q = (\mathrm{d} \P_j/\mathrm{d} \ol{\P}_\alpha) (\mathrm{d} \ol{\P}_\alpha/\mathrm{d} \Q)$, which entails
\[
\sum_{i=1}^N \alpha_i \, \KL(\P_i,\Q)
= \sum_{i=1}^N \alpha_i \bigintsss \left( \ln \frac{\d\P_i}{\d\ol{\P}_\alpha} +
\ln \frac{\d\ol{\P}_\alpha}{\d\Q}\right)\! \d \P_i
= \left( \sum_{i=1}^N \alpha_i \KL\bigl(\P_i,\ol{\P}_\alpha\bigr) \right) + \KL\bigl(\ol{\P}_\alpha,\Q\bigr) \,,
\]
where we used that $\displaystyle{\sum_{i=1}^N \alpha_i \d \P_i} = \d\ol{\P}_\alpha$.
So, indeed, the considered infimum is achieved at $\Q=\ol{\P}$. \\

\begin{proof}
The first inequality follows from the choice $\Q=\ol{\P}_\alpha$.
For the second inequality, we proceed as in Corollary~\ref{cor:jointconvKL--fdiv}
and consider the following probability distributions on $\{1,\ldots,N\} \times \Omega$:
for all $j \in \{1,\ldots,N\}$ and all $B \in \cF$,
\[
\tilde{\P}\bigl(\{j\} \times B \bigr) = \alpha_j \, \P_j(B)
\qquad \mbox{and} \qquad
\tilde{\Q}\bigl(\{j\} \times B \bigr) = \alpha_j \, \ol{\P}_\alpha(B)\,.
\]
Note that because $\alpha_i > 0$ for all $i$, we have $\P_j \ll \ol{\P}_\alpha$
for all $j$. Thus, $\tilde{\P} \ll \tilde{\Q}$, with Radon-Nikodym derivative given by
\[
(j,\omega) \in \{1,\ldots,N\} \times \Omega \,\, \longmapsto \,\,
\frac{\d\tilde{\P}}{\d\tilde{\Q}}(j,\omega) =
\frac{\d\P_j}{\d\ol{\P}_\alpha}(\omega)
\defeq p_j(\omega)\,.
\]
By uniqueness and linearity of the Radon-Nikodym derivatives,
we thus have, for $\ol{\P}_\alpha$--almost all $\omega$,
\[
\sum_{j=1}^N \alpha_j \, p_j(\omega)
= \sum_{j=1}^N \alpha_j \, \frac{\d\P_j}{\d\ol{\P}_\alpha}(\omega)
= \frac{\d\ol{\P}_\alpha}{\d\ol{\P}_\alpha}(\omega) = 1\,,
\qquad \mbox{where} \qquad \forall k \in \{1,\ldots,N\}, \ \ \alpha_k p_k(\omega) \geq 0\,;
\]
that is, $\alpha p(\omega) = \bigl( \alpha_j p_j(\omega) \bigr)_{1 \leq j \leq N}$ is
a probability distribution on $\{1,\ldots,N\}$. (It corresponds to the conditional distribution of $j$ given $\omega$ in the probabilistic model $j \sim \alpha$ and $\omega | j \sim \P_j$.)

We now compute $\Div\bigl(\tilde{\P},\tilde{\Q}\bigr)$ in two different ways.
All manipulations below are valid because all integrals defining $f$--divergences
exist (see the comments after the statement of Definition~\ref{def:fdiv},
as well as the first part of the proof of Lemma~\ref{lem:Jensen2}).
Integrating over $j$ first,
\begin{align*}
\Div\bigl(\tilde{\P},\tilde{\Q}\bigr)
& = \bigintss_{\{1,\ldots,N\} \times \Omega} f\!\left(\frac{\d \tilde{\P}}{\d \tilde{\Q}}(j,\omega)\right)\! \d \tilde{\Q}(j,\omega) \\
& = \sum_{j=1}^N \alpha_j \bigintsss_{\Omega} f\!\left(\frac{\d\P_j}{\d\ol{\P}_\alpha}(\omega)\right)\! \d \ol{\P}_\alpha(\omega)
= \sum_{j=1}^N \alpha_j \, \Div\!\left(\P_j,\ol{\P}_\alpha\right)\,.
\end{align*}
On the other hand,
integrating over $\omega$ first,
\begin{align*}
\Div\bigl(\tilde{\P},\tilde{\Q}\bigr) & =
\bigints_{\Omega} \left( \sum_{j=1}^N f\bigl(p_j(\omega)\bigr) \, \alpha_j \right)\! \d \ol{\P}_\alpha(\omega) \\
& = \bigints_{\Omega} \left( \sum_{j=1}^N f\!\left(\frac{\alpha_j p_j(\omega)}{\alpha_j} \right) \alpha_j \right)\!
\d \ol{\P}_\alpha(\omega) =
\bigintsss_{\Omega} \Div\bigl( \alpha p(\omega), \alpha \bigr) \d \ol{\P}_\alpha(\omega)
\leq B_f(\alpha)\,,
\end{align*}
where the last inequality
follows by noting that, by joint convexity of $\Div$ (see Corollary~\ref{cor:jointconvKL--fdiv}),
\[
\Div\bigl( \alpha p(\omega), \alpha \bigr)
\leq \sum_{j=1}^n \alpha_j p_j(\omega) \,
\Div\bigl( \delta_j, \alpha \bigr) \leq B_f(\alpha)\,.
\]
Comparing the two obtained expressions for $\Div\bigl(\tilde{\P},\tilde{\Q}\bigr)$ concludes the proof.
\end{proof}

\newpage
\section{On Birg{\'e}'s lemma}
\label{sec:Birge-HAL}

Theorem~\ref{th:Birge} is actually a slightly simplified version of the main result by~\citet{Birge05Fano} (his Corollary~1).
The proof of Theorem~\ref{th:Birge} below follows the methodology described in Section~\ref{sec:prooftech}.
In Section~\ref{sec:Birge-other}, we also state, discuss, and prove a previous (looser) simplification by~\citet{Massart03StFlour}
and the original result \citep[Corollary~1]{Birge05Fano}. 

\subsection{Proof of Theorem~\ref{th:Birge}}
\label{sec:app-proofBirge}

\begin{proof}
We denote by $h : p \in [0,1] \mapsto - \bigl( p \ln(p) + (1-p)\ln(1-p) \bigr)$
the binary entropy function. The existence of $c_N$ follows from the fact
that $c \in (0,1) \mapsto h(c)/c + \ln(1-c)$ is continuous and decreasing, as the sum of two
such functions; its respective limits are $+\infty$ and $-\infty$ at $0$ and $1$.

Reduction~\eqref{eq:red1} with $\Q_i = \P_1$ for all $i \geq 2$ indicates that
$\kl\bigl(\tilde{p},\tilde{q}\bigr) \leq \ol{K}$ where
\[
\tilde{p} \defeq \frac{1}{N-1} \sum_{i=2}^N \P_i(A_i)\,, \qquad
\tilde{q} \defeq \frac{1}{N-1} \sum_{i=2}^N \P_1(A_i) = \frac{1-\P_1(A_1)}{N-1}\,, \qquad
\ol{K} = \frac{1}{N-1} \sum_{i=2}^N \KL(\P_i,\P_1)\,;
\]
note that
we used the assumption of a partition to get the alternative definition of the $\tilde{q}$ quantity.
We use the following lower bound on $\kl$, which follows from calculations
similar to the ones performed in~\eqref{eq:bd1-intro}, using that $c_N \geq 1/2$ and that the binary
entropy $h : p \mapsto - \bigl( p \ln(p) + (1-p)\ln(1-p) \bigr)$ is decreasing on $[1/2,\,1]$:
for $p \geq c_N$,
\[
\kl(p,q) \geq p \ln\!\left(\frac{1}{q}\right) - h(c_N)
\geq p \ln\!\left(\frac{1}{q}\right) - p \, \frac{h(c_N)}{c_N} \,,
\]
where $\ln(1/q) - h(c_N)/c_N > 0$ for $q < \exp\bigl(-h(c_N)/c_N\bigr)$. Hence,
\begin{equation}
\label{eq:bd1-Birge}
\forall p \in [0,1], \ \
\forall \, q \in \Bigl(0, \, \exp\bigl(-h(c_N)/c_N\bigr) \Bigr), \qquad \quad
p \leq \max\!\left\{ c_N, \,\, \frac{\kl(p,q)}{\ln(1/q) - h(c_N)/c_N} \right\}.
\end{equation}
Now, we set $a = \displaystyle{\min_{1 \leq i \leq N} \P_i(A_i)}$ and may
assume $a \geq c_N$ (otherwise, the stated bound is obtained).

We have, by the very definition of $a$ as a minimum and by the definition~\eqref{def:cNBirge} of $c_N$,
\begin{equation}
\label{eq:Na-Birge}
a \leq \tilde{p} \qquad \mbox{and} \qquad
\tilde{q} \leq \frac{1-a}{N-1} \leq \frac{1-c_N}{N-1} = \frac{1}{N} \exp \! \left( - \frac{h(c_N)}{c_N} \right) \,.
\end{equation}
Note that, if $\tilde{q} = 0$ then $\ol{K} \geq \kl\bigl(\tilde{p},\tilde{q}\bigr) = +\infty$ (since $\tilde{p} \geq a \geq c_N > 0$) so that the desired bound holds trivially. We may therefore assume that $\tilde{q} > 0$ and combine $\kl\bigl(\tilde{p},\tilde{q}\bigr) \leq \ol{K}$
with~\eqref{eq:bd1-Birge} to get
\[
a \leq \tilde{p} \leq
\max\!\left\{ c_N, \,\, \frac{\kl\bigl(\tilde{p},\tilde{q}\bigr)}{\ln\bigl(1/\tilde{q}\bigr) - h(c_N)/c_N}
\right\} \leq \max\!\left\{ c_N, \,\, \frac{\ol{K}}{\ln(N)}
\right\},
\]
where, for the last inequality, we used the upper bound on $\tilde{q}$ in~\eqref{eq:Na-Birge}.
\end{proof}

\subsection{Two other statements of Birg{\'e}'s lemma}
\label{sec:Birge-other}

The original result by \citet[Corollary~1]{Birge05Fano}
reads, with the notation of Theorem~\ref{th:Birge}:
\begin{equation}
\label{eq:Birgeoriginal}
\min_{1 \leq i \leq N} \P_i(A_i) \leq \max\!\left\{ d_N,\,\,
\frac{\ol{K}}{\ln(N)}
\right\},
\end{equation}
where $(d_N)_{N \geq 2}$ is a decreasing sequence, defined as follows,
based on functions $r_N : [0,1) \to \R$:
\[
r_N(b) = \kl\biggl(b,\frac{1-b}{N-1}\biggr) - b \ln(N)
\qquad \mbox{and} \qquad
d_N = \max \bigl\{ b \in [0,1] : \ r_N(b) \leq 0 \bigr\}\,.
\]
This original result was only stated for $N \geq 3$ but its proof
indicates that it is also valid for $N=2$.

On the other hand, the simplification by \citet[Section~2.3.4]{Massart03StFlour} leads to
\begin{equation}
\label{eq:BirgeMassart}
\min_{1 \leq i \leq N} \P_i(A_i) \leq \max\!\left\{ \frac{2\e - 1}{2\e}, \,\,
\frac{\ol{K}}{\ln(N)}
\right\}.
\end{equation}
(The original constant was a larger $2\e/(2\e+1)$ in \citealp[Section~2.3.4]{Massart03StFlour}.)
\medskip

Before proving these results, we compare them with Theorem~\ref{th:Birge}.
The values of the $c_N$ of Theorem~1, of the $d_N$ of~\eqref{eq:Birgeoriginal}
and of $(2\e-1)/(2\e)$ are given by (values rounded upwards)
\begin{center}
$\displaystyle{\frac{2\e-1}{2\e}} \approx 0.8161$ \qquad \qquad \mbox{and} \qquad \qquad
\begin{tabular}{cccccc}
\hline
$N$ & \ & $2$ & $3$ & $7$ & $+\infty$ \\
$c_N$ & & $0.7587$ & $0.7127$ & $< 0.67$ & $0.63987$ \\
$d_N$ & & $0.7428$ & $0.7009$ & $< 2/3$ & $0.63987$ \\
\hline
\end{tabular}
\end{center}
The $c_N$ and $d_N$ are thus extremely close. While the $c_N$
are slightly larger than the $d_N$ (with, however, the same limit),
they are easier to compute in practice. (See the closed-form expression for $r_N$
below.) Also, the proof of Theorem~\ref{th:Birge}
is simpler than the proof of \citet[Corollary~1]{Birge05Fano}: they rely
on the same proof scheme but the former involves fewer calculations than the latter.
Indeed, let us now prove again \citet[Corollary~1]{Birge05Fano}.

\paragraph{Proof of~\eqref{eq:Birgeoriginal}:}
We use the notation of the proof of Theorem~\ref{th:Birge} and
its beginning.
We can assume with no loss of generality that $a \geq d_N$, and we also have $d_N \geq 1/N$ as $r_N(1/N)=-\ln(N)/N \leq 0$. Therefore, $a \geq 1/N$ and
using the definition of $a$ as a minimum,
\begin{equation}
\label{eq:classmtB}
\tilde{q} \leq \frac{1-a}{N-1} \leq a \leq \tilde{p}\,;
\end{equation}
therefore,
\[
\kl\bigl(\tilde{p},\tilde{q}\bigr)
\geq \kl\bigl(a,\tilde{q}\bigr)
\geq \kl\biggl(a,\frac{1-a}{N-1}\biggr)\,,
\]
since by convexity, $p \mapsto \kl(p,q)$ is
increasing on $[q,1]$ and $q \mapsto \kl(p,q)$ is
decreasing on $[0,p]$.
Combining this with $\ol{K} \geq \kl\bigl(\tilde{p},\tilde{q}\bigr)$,
one has proved
\[
\ol{K} \geq \kl\biggl(a,\frac{1-a}{N-1}\biggr)
= a \ln(N) + r_N(a)\,,
\]
from which the bound~\eqref{eq:Birgeoriginal} follows by definition of $d_N$.
To prove that the sequence $(d_N)$ is decreasing and to get a numerical expression
via dichotomy follow from studying the variations of $r_N(b)$ in $b$ and~$N$;
for the latter; one should show, in particular, that $r_N(b)$ is positive before $d_N$
and negative after $d_N$. This last analytical part of the proof is tedious, as
\begin{align*}
r_N(a) & = a \ln \!\left(\frac{a}{1-a} \right) + (1-a) \ln \!\left( \frac{1-a}{1 - \displaystyle{\frac{1-a}{N-1}}} \right) + \bigl( a \ln(N-1) - a \ln(N) \bigr) \\
& = \bigl( a \ln(a) + (1-2a)\ln(1-a) \bigr) + a \ln\!\left(\frac{N-1}{N}\right)
+ (1-a) \ln\!\left(\frac{N-1}{N-2+a}\right),
\end{align*}
and we could overcome these heavy calculations in our proof of  Theorem~\ref{th:Birge}. \qed 

\paragraph{Proof of~\eqref{eq:BirgeMassart}:}
For $p \geq \ln(2)$ and all $q \in [0,1]$,
\begin{equation}
\label{eq:lbklmassart}
\kl(p,q) \geq p \ln\!\left(\frac{1}{q}\right) - \ln(2)
\geq p \ln\!\left(\frac{1}{q}\right) - p
= p \ln\!\left(\frac{1}{\e q}\right).
\end{equation}
Equation~\eqref{eq:Na-Birge} is adapted as
\[
a \leq \tilde{p} \qquad \mbox{and} \qquad
\tilde{q} \leq \frac{1-a}{N-1} \leq
\frac{2(1-a)}{N} \leq \frac{1}{\e N}
\]
where we used respectively, for the last two inequalities,
that $1/(N-1) \leq 2/N$ for $N \geq 2$
and that, with no loss of generality,
$a \geq (2\e-1)/(2\e)$.
In particular, $\e \tilde{q} \leq 1/N$.
Combining this with
$\ol{K} \geq \kl\bigl(\tilde{p},\tilde{q}\bigr)$
and~\eqref{eq:lbklmassart}, we have proved
\[
\ol{K} \geq
\tilde{p} \, \ln\!\left(\frac{1}{\e \tilde{q}}\right)
\geq a \ln(N)\,,
\]
which concludes the proof. \qed